\documentclass[12pt,a4paper]{amsart}
\usepackage{amsmath,amsfonts,amssymb,amsthm,amscd}
\usepackage[T1]{fontenc}
\usepackage{ae,aecompl}
\usepackage[mathscr]{eucal}
\usepackage{mathrsfs}
\usepackage[english]{babel}
\usepackage{latexsym}
\usepackage[all]{xy}
\usepackage{cite}
\usepackage{tikz}
\usepackage{tikz-cd}
\usepackage{textcomp}
\usetikzlibrary{graphs}
\usetikzlibrary{calc}
\usepackage[english]{isodate}
\usepackage[colorlinks=false]{hyperref}
\hypersetup
{
    pdfauthor={Seung-Jo Jung},
    pdftitle={On the Craw--Ishii Conjecture}
}

\theoremstyle{plain}
\newtheorem{Thm}[equation]{Theorem}
\newtheorem{Prop}[equation]{Proposition}
\newtheorem{Cor}[equation]{Corollary}
\newtheorem{Lem}[equation]{Lemma}
\newtheorem{Con}[equation]{Conjecture}

\theoremstyle{definition}

\newtheorem{Def}[equation]{Definition}
\newtheorem{Que}[equation]{Question}

\newtheorem{Eg}[equation]{Example}

\newtheorem{Convention}[equation]{Convention}
\newtheorem{Rem}[equation]{Remark}
\numberwithin{equation}{section}
\numberwithin{figure}{section}
\numberwithin{table}{section}

\newcommand{\C}{\mathbb C} 
 
\newcommand{\Q}{\mathbb Q} 
\newcommand{\R}{\mathbb R} 
 
\newcommand{\Z}{\mathbb Z} 

\newcommand{\bu}{\mathbf u} 
\newcommand{\bk}{\mathbf n} 
\newcommand{\bm}{\mathbf m} 
\newcommand{\bn}{\mathbf p} 

\newcommand{\Cone}{\ensuremath{\operatorname{Cone}}}

\newcommand{\eqd}{\stackrel{\rm \tiny def}{=}} 

\newcommand{\sF}{\mathcal F} 
\newcommand{\sG}{\mathcal G} 
\newcommand{\sH}{\mathcal H} 
\newcommand{\sM}{\mathcal M} 

\newcommand{\OZ}{{\mathcal O}\!_{Z}} 
\newcommand{\KY}{{K}\!_{Y}} 
\newcommand{\KXcan}{{K}\!_{X_{\mathrm{can}}}} 

\newcommand{\dual}{^{\vee}} 
\newcommand{\mul}{^{\times}} 

\newcommand{\AHilb}{\ensuremath{A\text{-}\operatorname{\! Hilb}}}
\newcommand{\GHilb}[1]{\ensuremath{G\text{-}\operatorname{\! Hilb}}\C^{#1}}
\newcommand{\GHii}[1]{\ensuremath{G_{#1}\text{-}\operatorname{\! Hilb}}\C^3}
\newcommand{\GHil}{\ensuremath{G\text{-}\operatorname{\! Hilb}}}
\newcommand{\HHilb}{\ensuremath{\text{-}\operatorname{\! Hilb}}}

\newcommand{\Spec}{\ensuremath{\operatorname{Spec}}}

\newcommand{\Hom}{\ensuremath{\operatorname{Hom}}}
\DeclareMathOperator{\GL}{GL}
\DeclareMathOperator{\SL}{SL}
\DeclareMathOperator{\wt}{wt} 
\DeclareMathOperator{\Irr}{Irr} 

\newcommand{\diag}{\ensuremath{\operatorname{diag}}}

\newcommand{\iso}{\cong}

\newcommand{\st}{\ensuremath{\operatorname{\bigm{|}}}}
\newcommand{\stB}{\ensuremath{\operatorname{\bigg{|}}}}

\newcommand{\QED}{\hfill$\blacksquare$\end{proof}\vskip 3pt}
\newcommand{\eeg}{~\hfill$\lozenge$\end{Eg}\vskip 3pt }
\newcommand{\erem}{~\hfill$\lozenge$\end{Rem}\vskip 3pt }
\newcommand{\ewar}{~\hfill$\lozenge$\end{War}\vskip 3pt }

\newcommand{\pull}{^{\ast}}

\newcommand{\inv}{^{-1}}

\newcommand{\hh}{\ensuremath{\operatorname{H}}} 

\newcommand{\mon}{\overline{M}_{\geq 0}} 
\newcommand{\lau}{\overline{M}} 
\newcommand{\wtga}{\wt_{\Gamma}} 
\newcommand{\wtgp}{\wt_{\Gamma'}}
\newcommand{\mth}{\sM_{\theta}} 
\newcommand{\yth}{Y_{\theta}} 

\newcommand{\lf}{\lfloor}
\newcommand{\rf}{\rfloor}

\newcommand{\one}{\boldsymbol{1}}

\newcommand{\Homz}{\ensuremath{\operatorname{Hom}}_{\mathbb Z}}
\newcommand{\zzz}{\overline{L}}

\newcommand{\mco}{{\Sigma_{\mathrm{max}}}} 
\newcommand{\gr}{{\mathfrak{S}}} 

\newcommand{\vtheta}{\psi} 

\newcommand{\roundstab}{(\phi_k)_{\star}}
\newcommand{\roundbrick}{\phi_k^{\star}}

\title[On the Craw--Ishii conjecture]{On the Craw--Ishii conjecture}

\author[S.-J. Jung]{Seung-Jo Jung}
\address{Korea Institute for Advanced Study, 85 Hoegiro, Dongdaemun-gu, Seoul, 130-722, Republic of Korea}
\email{seungjo@kias.re.kr}

\date{\today}

\begin{document}


\begin{abstract}
In~\cite{CI}, Craw and Ishii proved that for a finite abelian group $G\subset \SL_3(\C)$ every (projective) relative minimal model of $\C^3/G$ is isomorphic to the fine moduli space $\mth$ of $\theta$-stable $G$-constellations for some GIT parameter $\theta$. In this article, we conjecture that the same is true for a finite group $G\subset \GL_3(\C)$ if a relative minimal model $Y$ of $X=\C^3/G$ is smooth. 
We prove this for some abelian groups.
\end{abstract}
\maketitle

\tableofcontents 
\section{Introduction}
Let $G$ be a finite subgroup in $\GL_n(\C)$. A {\em $G$-cluster} $Z$ is a $G$-invariant subscheme of $\C^n$ with $\hh^0(\OZ)$ isomorphic to $\C[G]$ the regular representation of $G$.
For a finite group $G$ in $\SL_2(\C)$, Ito and Nakamura\cite{INa} showed that the minimal resolution of $\C^2/G$ is isomorphic to the \emph{$G$-Hilbert scheme} $\GHilb{2}$ that is the fine moduli space of $G$-clusters.
In \cite{BKR}, Bridgeland, King and Reid proved that for a finite group $G$ in $\SL_3(\C)$ the $G$-Hilbert scheme $\GHilb{3}$ is a crepant resolution of $\C^3/G$.

In\cite{CI}, Craw and Ishii introduced a generalised notion of $G$-clusters. A {\em $G$-constellation} $\sF$ is a $G$-equivariant sheaf on $\C^n$ with $\hh^0(\sF)$ isomorphic to $\C[G]$. Define the GIT stability parameter space
\[
\Theta = \left\{\theta \in \Hom_{\Z} (R(G),\Q) \st \theta\left(\C[G]\right) =0 \right\}\!,
\]
where $R(G)$ denotes the representation space of $G$. For $\theta\in\Theta$, we say that $G$-constellation $\sF$ is \emph{$\theta$-(semi)stable} if $\theta(\sG)> 0\ (\theta(\sG)\geq 0)$ for every nonzero proper subsheaf $\sG$ of $\sF$.
A stability parameter $\theta$ is called {\em generic} if every $\theta$-semistable $G$-constellation is $\theta$-stable. 

Furthermore, Craw and Ishii constructed the moduli space $\mth$ of $\theta$-stable $G$-constellations using GIT\cite{CI}. They conjectured that for a finite subgroup $G\subset\SL_3(\C)$, every projective crepant resolution of $\C^3/G$ is isomorphic to $\mth$ for some $\theta \in \Theta$ and proved this for $G$ being abelian. Note that if $G\subset\SL_3(\C)$, then $\C^3/G$ has Gorenstein canonical singularities.
Being motivated by this, this article makes the following conjecture and proves the conjecture for some cases.
\begin{Con}[Craw--Ishii conjecture]\label{Con:Craw-Ishii conjecture}
Let $G$ be a finite subgroup in $\GL_3(\C)$. Suppose $X=\C^3/G$ has a smooth relative minimal model. Then every relative minimal model of $X$ is isomorphic to (an irreducible component of) $\mth$ for a suitable GIT parameter $\theta$.
\end{Con}

On the other hand, $\mth$ need not be irreducible in general\cite{CMTb}. However, if $G$ is abelian, Craw, Maclagan and Thomas~\cite{CMT} showed that $\mth$ has a unique irreducible component $\yth$ containing the torus $(\C\mul)^n/G$ for generic $\theta$. Furthermore, they proved that $\yth$ can be obtained by variation of GIT from $\C^n/G$. The component $\yth$ is called the \emph{birational component} of $\mth$. 

\begin{Thm}[Main Theorem]\label{Intro_a:main theorem}
Let $G\subset \GL_3(\C)$ be the finite group of type $\frac{1}{r}(1,a,b)$ with $b$ coprime to $a$ satisfying one of the following:
\begin{enumerate}
\item $r=abc+a+b+1$ for some positive integer $c$;
\item $r=abc+a-2b+1$ and $b=ak+1$ for some positive integers $c,k$ with $c\geq 2$ and $a \geq 3$.
\end{enumerate}
Then every relative minimal model $Y\rightarrow X:=\C^3/G$ is isomorphic to the birational component $\yth$ of the moduli space $\mth$ of $\theta$-stable $G$-constellations for a suitable parameter~$\theta$.
\end{Thm}

Moreover the main theorem implies that the relative minimal model can be obtained by variation of GIT from $\C^3/G$ for the cases.

\begin{Cor}
In the situation as in Theorem~\ref{Intro_a:main theorem}, every relative minimal model $Y\rightarrow X=\C^3/G$ is obtained by variation of GIT quotient.
\end{Cor}

\subsection{Overview of the article}
Let $G$ be a finite group in $\GL_3(\C)$. We say that $\varphi\colon Y\to X:=\C^3/G$ is a \emph{relative minimal model} if:
\begin{enumerate}
\item $Y$ has only $\Q$-factorial terminal singularities;
\item $\KY$ is $\varphi$-nef;
\item $\varphi$ is projective.
\end{enumerate}
For example, for the case where $G\subset \SL_3(\C)$, a projective crepant resolution $Y \to \C^3/G$ is a relative minimal model.

For the group $G$ of type $\frac{1}{r}(\alpha_1,\ldots,\alpha_n)$, i.e.
\[
G=\langle\diag(\epsilon^{\alpha_1},\ldots,\epsilon^{\alpha_n}) \st \epsilon^r=1 \rangle \subset \GL_n(\C),
\]
by toric geometry, the quotient variety $X=\C^n/G$ is given by the toric cone
\[
\sigma_+:=\Cone(e_1,\ldots,e_n)
\]
with the lattice 
\begin{equation*}
L = \Z^n + \Z\cdot \frac{1}{r}(\alpha_1,\ldots,\alpha_n).
\end{equation*}
Fix a primitive interior lattice point $v=\tfrac{1}{r}(a_1,\ldots,a_n) \in L \cap \sigma_+$.
The {\em star subdivision} of $\sigma_+$ at $v$ is the minimal fan containing the following $n$-dimensional cones $\sigma_k$ for $k=1,\ldots,n$:
\[
\sigma_k:=\Cone(e_1,\ldots,\hat{e_k},v,\ldots,e_n).
\]
Then the corresponding toric variety $X_v$ admits the induced projective toric morphism $\nu \colon X_v \to X=\C^n/G$. If $v$ generates $L/\Z^n$, then the affine open set $U_k$ of $X_v$ corresponding to $\sigma_k$ has a quotient singularity $\C^n/G_k$ for some abelian group $G_k$, eg.\ $G_1$ is the group of type $\frac{1}{a_1}(-r,a_2,\ldots,a_n)$. Note that the order of $G_k$ is smaller than the order of $G$. Thus we can use induction on the order of groups.

The groups in the main theorem satisfy:
\begin{enumerate}
\item every relative minimal model $Y\to X$ is smooth;
\item every relative minimal model $Y\to X$ has a projective morphism $Y \to X_v$ for $v=\frac{1}{r}(1,a,b)$.
\end{enumerate}

For the proof, the notion of $G$-bricks and round down functions is essential, which was recently developed in \cite{Thesis, Terminal}.

A \emph{$G$-brick} $\Gamma$ is a certain $\C$-basis of $\hh^0(\sF)$ for a torus invariant $G$-constellation $\sF$ on the birational component (see Definition~\ref{Def:G-prebrick}). We say that $\Gamma$ is \emph{$\theta$-stable} if the corresponding $G$-constellation $\sF$ is $\theta$-stable. Using suitable $G$-bricks, we are able to describe an affine local chart of the birational component $\yth$ (see Theorem~\ref{Thm:Y theta with G-bricks}).

The \emph{round down functions for the star subdivisions at $v$} are maps between monomial lattices compatible with the star subdivision. Using the round down functions, we produce a set $\gr$ of $G$-bricks from $G_k$-bricks.

Since the set $\gr$ is a \emph{$G$-brickset} (see Definition~\ref{Def:brickset}), it suffices to find a GIT parameter $\theta$ such that every $G$-brick $\Gamma \in \gr$ is $\theta$-stable. After finding a parameter $\theta$, we conclude that $Y$ is isomorphic to $\yth$ for some~$\theta$.

\subsubsection*{\bf Acknowledgement} I am deeply grateful to Miles Reid for his valuable advice and encouragement. I would like to thank Alastair Craw, Akira Ishii, Yukari Ito, Yujiro Kawamata 
for kind explanations. I am grateful to Sara Muhvi\'{c} for sharing her examples. I would like to thank the University of Warwick for its hospitality, where this work was done in June 2015.
\section{$G$-constellations and $G$-bricks}\label{Sec:G-constellations and G-bricks}

\subsection{Moduli spaces of $G$-constellations}
In this section, we briefly review moduli spaces of $\theta$-stable $G$-constellations (see e.g. \cite{CI,CMT,K94}).

Consider a finite diagonal group $G$ in $\GL_n(\C)$. 
\begin{Def}
A $G$-equivariant coherent sheaf $\sF$ on $\C^n$ is called a {\em $G$-constellation} if $\hh^0(\sF)$ is isomorphic to the regular representation~$\C[G]$ of $G$ as a $\C[G]$-module. 
\end{Def}
\begin{Rem}
For a free $G$-orbit $Z$ in $\C^3$, $\OZ$ is a $G$-constellation.
\erem
Define the GIT stability parameter space
\[
\Theta = \left\{\theta \in \Hom_{\Z} (R(G),\Q) \st \theta\left(\C[G]\right) =0 \right\}\!
\]
where $R(G):=\bigoplus_{\rho \in \Irr G} \Z\cdot \rho$ is the representation space of $G$.
\begin{Def}
For a stability parameter $\theta \in \Theta$, we say that:
\begin{enumerate}
\item a $G$-constellation $\sF$ is {\em $\theta$-semistable} if $\theta(\sG) \geq 0$ for every subsheaf $\sG \varsubsetneq \sF$;
\item a $G$-constellation $\sF$ is {\em $\theta$-stable} if $\theta(\sG) > 0$ for every subsheaf $0\neq \sG \varsubsetneq \sF$;
\item $\theta$ is {\em generic} if every $\theta$-semistable object is $\theta$-stable.
\end{enumerate}
\end{Def}
By King\cite{K94}, it is known that if $\theta$ is generic, then there exists a quasiprojective scheme $\mth$ which is a fine moduli space of $\theta$-stable $G$-constellations.

Moreover, in~\cite{IN}, Ito--Nakajima showed that $\mth$ is canonically isomorphic to $\GHilb{n}$ for $\theta\in \Theta_+$ where 
\begin{equation}\label{Eqtn:Stab for G-Hilb}
\Theta_+ := \left\{\theta \in \Theta \st \theta\left(\rho\right) >0 \text{ for } \rho \neq \rho_0 \right\}.
\end{equation}
In particular, $\mth$ can be obtained by variation of GIT from $\GHilb{n}$.

Assume that $\theta$ is generic. Let $\mth$ denote the fine moduli space of $\theta$-stable $G$-constellations. Craw, Maclagan and Thomas showed that the moduli space $\mth$ need not be irreducible~\cite{CMTb}. Furthermore, they proved that $\mth$ has a distinguished component $\yth$ which is birational to $\C^n/G$ if $G$ is abelian~\cite{CMT}.
\begin{Thm}[Craw--Maclagan--Thomas{\cite{CMT}}]\label{Thm:CMT}
Assume that $G$ be a finite abelian group in $\GL_n(\C)$. For a generic parameter $\theta \in \Theta$, the moduli space $\mth$ has a unique irreducible component $\yth$ that contains the torus $T:=(\C\mul)^n/G$. Moreover:
\begin{enumerate}
\item $\yth$ is a not-necessarily-normal toric variety which is birational to the quotient variety $\C^n/G$;
\item there is a projective morphism $\yth \to \C^n/G$ obtained by variation of GIT quotient.
\end{enumerate}
\end{Thm}
\begin{Def}
The unique irreducible component $\yth$ in Theorem~\ref{Thm:CMT} is called the {\em birational component} of $\mth$. 
\end{Def}
\subsection{Cyclic quotients and toric lattices}
Consider the group $G$ of type $\frac{1}{r}(\alpha_1,\ldots,\alpha_n)$, i.e.
\[
G=\langle\diag(\epsilon^{\alpha_1},\ldots,\epsilon^{\alpha_n}) \st \epsilon^r=1 \rangle \subset \GL_n(\C).
\]
As $G$ is abelian, the set of irreducible representations of~$G$ can be identified with the character group~$G\dual:=\Hom (G, \C\mul)$ of~$G$.

For the group $G$ of type $\frac{1}{r}(\alpha_1,\ldots,\alpha_n)$, define the lattice 
\begin{equation*}
L = \Z^n + \Z\cdot \frac{1}{r}(\alpha_1,\ldots,\alpha_n).
\end{equation*}
Set $\zzz = \Z^n \subset L$. Consider the two dual lattices $M=\Homz(L,\Z)$, $\lau=\Homz(\zzz,\Z)$. Note that we can consider the two dual lattices $\lau$ and $M$ as Laurent monomials and $G$-invariant Laurent monomials, respectively. 
 
The embedding of~$G$ into the torus $(\C\mul)^n\subset \GL_n(\C)$ induces a surjective homomorphism
\[
\wt \colon \lau \longrightarrow G\dual
\]
with kernel $M$. 
Note that there are two isomorphisms of abelian groups $L/\Z^n \rightarrow G$ and $\lau / M \rightarrow G\dual$.

Let $\mon$ denote genuine monomials in $\lau$, i.e.\
\begin{equation*}
\mon=\{x_1^{m_1}\cdots x_n^{m_n} \in \lau \st m_i \geq 0 \text{ for all $i$}\}.
\end{equation*}
For a set $A \subset \C[x_1^{\pm},\ldots,x_n^{\pm}]$, $\langle A \rangle$ denotes the $\C[x_1,\ldots,x_n]$-submodule of $\C[x_1^{\pm},\ldots,x_n^{\pm}]$ generated by $A$.

Let $\{e_1,\ldots,e_n\}$ be the standard basis of $\Z^n$ and $\sigma_{+}$ the cone generated by $e_1,\ldots,e_n$. By toric geometry, the corresponding affine toric variety 
$U_{\sigma_{+}}=\Spec \C[\sigma_{+}\dual \cap M]$ is the quotient variety $X=\C^3/G$.

\subsection{$G$-bricks and the birational component $\yth$}
In this section, we review the notion of $G$-bricks introduced in \cite{Thesis, Terminal}. Using $G$-bricks, we can describe an affine local chart of the birational component $\yth$.

\begin{Def}\label{Def:G-prebrick}\index{$G$-prebrick}
A {\em $G$-prebrick} $\Gamma$ is a subset of Laurent monomials in $\C[x_1^{\pm},\ldots,x_n^{\pm}]$ satisfying:
\begin{enumerate}
\item the monomial $\one$ is in $\Gamma$;
\item for each weight $\rho \in G\dual$, there exists a unique Laurent monomial $\bm_{\rho}\in \Gamma$ of weight $\rho$, i.e.\ $\wt\colon\Gamma \rightarrow G\dual$ is bijective;
\item if $\bn' \cdot \bn \cdot \bm_{\rho} \in \Gamma$ for $\bm_{\rho}\in \Gamma$ and $\bn, \bn' \in \mon$, then $\bn \cdot \bm_{\rho} \in \Gamma$;
\item the set $\Gamma$ is {\em connected} in the sense that for any element $\bm_{\rho}$, there is a (fractional) path in $\Gamma$ from $\bm_{\rho}$ to $\one$  whose steps consist of multiplying or dividing by one of $x_i$.
\end{enumerate}
\end{Def}
For a Laurent monomial $\bm \in \lau$, let $\wtga(\bm)$ denote the unique element $\bm_{\rho}$ in $\Gamma$ of the same weight as $\bm$.

For a $G$-prebrick $\Gamma=\{\bm_{\rho}\}$, we define $S(\Gamma)$ to be the subsemigroup of~$M$ generated by $\dfrac{\bn \cdot \bm_{\rho}}{\wtga(\bn \cdot \bm_{\rho})}$ for all $\bn \in \mon$, $\bm_{\rho}\in \Gamma$. We define a cone $\sigma(\Gamma)$ in $L_{\R}=\R^n$ as follows:
\begin{align*}
\sigma(\Gamma)
&=S(\Gamma)\dual\\
&=\left\{ \bu \in L_{\R} \stB  \left\langle \bu, \frac{\bn \cdot \bm_{\rho}}{\wtga(\bn \cdot \bm_{\rho})}\right\rangle \geq 0, \quad \forall \bm_{\rho}\in \Gamma, \ \bn \in \mon \right\}\!.
\end{align*}
As is proved in \cite{Terminal}, the semigroup $S(\Gamma)$ is finitely generated as a semigroup. Thus the semigroup~$S(\Gamma)$ defines an affine toric variety.
Define two affine toric varieties:
\begin{align*}
U(\Gamma) &:= \Spec \C[S(\Gamma)],\\
U^{\nu}(\Gamma)&:= \Spec \C[\sigma(\Gamma)\dual\cap M]. 
\end{align*}
\begin{Def}\label{Def:border bases}
For a $G$-prebrick $\Gamma$,
\[
B(\Gamma) := \big\{ x_i \cdot \bm_{\rho} \st \bm_{\rho} \in \Gamma,
\big\} \!\setminus \!\Gamma
\]
is called the {\em Border bases} of $\Gamma$.
\end{Def}
Let $\Gamma$ be a $G$-prebrick. Define 
\[
C(\Gamma) := \langle \Gamma \rangle / \langle B(\Gamma)\rangle.
\]
The module $C(\Gamma)$ is a torus invariant $G$-constellation. A submodule $\sG$ of $C(\Gamma)$ is determined by a subset $A \subset \Gamma$, which forms a $\C$-basis of~$\sG$.
\begin{Lem}\label{Lem:combinatorial description of submodule}
Let $A$ be a subset of $\Gamma$. The following are equivalent.
\begin{enumerate}
\item The set $A$ forms a $\C$-basis of a submodule of $C(\Gamma)$.
\item If $\bm_{\rho} \in A$, then $x_i \cdot \bm_{\rho} \in \Gamma$ implies $x_i \cdot \bm_{\rho} \in A$ for all $i$.
\end{enumerate}
\end{Lem}

\begin{Def}
Let $\Gamma$ be a $G$-prebrick.
\begin{enumerate}
\item A $G$-prebrick $\Gamma$ is called a {\em $G$-brick} if the affine toric variety $U(\Gamma)$ contains a torus fixed point.
\item A $G$-prebrick $\Gamma$ is called {\em $\theta$-stable} if the torus invariant $G$-constellation $C(\Gamma)$ is $\theta$-stable.
\end{enumerate}
\end{Def}
Note that from toric geometry, $U(\Gamma)$ has a torus fixed point if and only if $S(\Gamma) \cap (S(\Gamma))\inv=\{\one\}$, i.e.\ the cone $\sigma(\Gamma)$ is an $n$-dimensional cone.
\begin{Prop}[\cite{Thesis}]\label{Prop:open immersion}
For generic $\theta$, let $\Gamma$ be a $\theta$-stable $G$-brick and $\yth$ the birational component of~$\mth$. There exists an open immersion
\[
\begin{array}{ccc}
U(\Gamma)=\Spec \C[S(\Gamma)]  & \hookrightarrow& \yth.
\end{array}
\]
\end{Prop}
\begin{Rem}
The $G$-brick $\Gamma$ forms a $\C$-basis of $G$-constellations parametrised by $U(\Gamma)$.\erem
\begin{Thm}[\cite{Thesis}]\label{Thm:Y theta with G-bricks}
Let $G\subset\GL_n(\C)$ be a finite diagonal group and $\theta$ a generic GIT parameter for $G$-constellations. Assume that $\gr$ is the set of all $\theta$-stable $G$-bricks.
\begin{enumerate}
\item The birational component $\yth$ of $\mth$ is isomorphic to the not-necessarily-normal toric variety  $\bigcup_{\Gamma \in \gr} U(\Gamma)$. 
\item The normalisation of $\yth$ is isomorphic to the normal toric variety whose toric fan consists of the $n$-dimensional cones $\sigma(\Gamma)$ for $\Gamma \in \gr$ and their faces.
\end{enumerate} 
\end{Thm}
\begin{Rem}
For given $G$ and $\theta$, it is difficult to find all $\theta$-stable $G$-bricks in general.\erem
\begin{Def}\label{Def:brickset}
For a finite diagonal group $G\subset\GL_n(\C)$, assume that $Y$ is a normal toric variety admitting a proper birational morphism $Y \to X:=\C^n/G$. Let $\mco$ denote the set of the $n$-dimensional cones in the fan of $Y$. A set $\gr$ of $G$-bricks is called a \emph{$G$-brickset} for $Y$ if $\gr$ satisfies:
\begin{enumerate}
\item there is a bijective map $\mco \to \gr$ sending $\sigma$ to $\Gamma_{\sigma}$;
\item $S(\Gamma_{\sigma})=\sigma\dual \cap M$.
\end{enumerate}
\end{Def}
\begin{Prop}\label{Prop:if we have a brickset enough to show there exists theta}
Suppose that $Y\to X:=\C^n/G$ is a proper birational morphism. Let $\gr$ be a $G$-brickset for $Y$. Then $Y$ is isomorphic to the toric variety $\bigcup_{\Gamma \in \gr} U(\Gamma)$. Moreover, if there exists $\theta\in\Theta$ such that every $\Gamma$ in $\gr$ is $\theta$-stable, then $Y$ is isomorphic to $\yth$.
\end{Prop}
\begin{proof}
By definition, it is clear that $Y$ is isomorphic to the toric variety $\bigcup_{\Gamma \in \gr} U(\Gamma)$. Assume that there exists $\theta\in\Theta$ such that every $\Gamma$ in $\gr$ is $\theta$-stable. From Proposition~\ref{Prop:open immersion}, we can conclude that there exists an open immersion $\iota\colon Y\hookrightarrow \yth$. Furthermore, since $Y\to X$ is proper and $\yth \to X$ is projective, $\iota$ is a closed embedding between $n$-dimensional toric varieties. Thus $\iota$ is an isomorphism.
\end{proof}
\section{Star subdivisions and moduli descriptions}\label{Sec:Star subdivisions and moduli descriptions}
\subsection{Star subdivisions and round down functions}\label{Sec:star subdivisions and round down functions}
Fix a primitive lattice point $v=\tfrac{1}{r}(a_1,\ldots,a_n) \in L \cap \sigma_+$.
The {\em star subdivision} (or {\em barycentric subdivision}) $\Sigma$ of $\sigma_+$ at $v$ is the minimal fan containing all cones $\Cone(\tau,v)$ where $\tau$ varies over all faces of $\sigma_+$ with $v \not \in \tau$.
Let $X:=U_{\sigma_+}$ be the affine toric variety corresponding to $\sigma_+$ and $X_{\Sigma}$ the toric variety corresponding to the fan $\Sigma$. Then the star subdivision induces a projective toric morphism $\nu \colon X_{\Sigma} \rightarrow X=\C^n/G$ with the ramification formula
\begin{equation}\label{Eqtn:ramification formula for star subdivision}
r K_{X_{\Sigma}}-\nu\pull r K_X \equiv_{\mathrm{num}} (\sum_i a_i-r) E_v,
\end{equation}
where $E_v$ is the torus invariant prime divisor corresponding to the 1-dimensional cone $\Cone(v)$.

The fan $\Sigma$ consists of the $n$-dimensional cone $\sigma_k$ and its faces for $k=1,\ldots,n$:
\[
\sigma_k:=\Cone(e_1,\ldots,\hat{e_k},v,\ldots,e_n).
\]
\begin{figure}[h]
\begin{center}
\begin{tikzpicture}
\coordinate [label=left:$e_3$] (e3) at (0,0);
\coordinate [label=right:$e_2$] (e2) at (7,0);
\coordinate [label=right:$e_1$] (e1) at (3.5,6.06);
\coordinate [label=left:$\bf 0$] (0) at (-1,2);
\coordinate (v1) at (2.5,1.7);

\draw[fill] (v1) circle [radius=0.05];
\draw[fill] (e3) circle [radius=0.05];
\draw[fill] (e2) circle [radius=0.05];
\draw[fill] (e1) circle [radius=0.05];
\draw[fill] (0) circle [radius=0.05];
\node [above right] at (v1) {$v=\frac{1}{r}(a_1,a_2,a_3)$};
\draw (e3) -- (e2);
\draw (e3) -- (e1);
\draw (e2) -- (e1);
\draw (e3) -- (v1);
\draw (e2) -- (v1);
\draw (v1) -- (e1);
\foreach \x in {1,2,3}
    \draw (0) -- (e\x); 
\draw (0) -- (v1);
\node [below] at (2.5,1.1) {$\sigma_{1}$};
\node  at (2.2,3) {$\sigma_{2}$};
\node  at (4,3) {$\sigma_{3}$};
\end{tikzpicture}
\end{center}
\caption{Star subdivision $\Sigma$ of $\sigma_+$ at $v$}\label{Fig:Star subdivision of sigma_+ at v}
\end{figure}
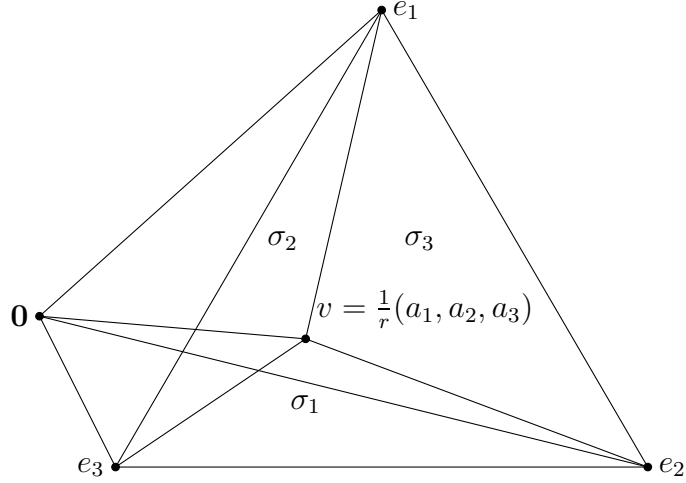
Assume that $v$ generates $L/\Z^n$. Fix $k\in\{1,\ldots,n\}$. Let $L_{k}$ be the sublattice of $L$ generated by $e_1,\ldots,\hat{e_k},v,\ldots,e_n$. Let us consider the dual lattice $M_{k} := \Homz(L_{k},\Z)$ with the corresponding dual basis~$\{\xi_1,\ldots, \xi_n\}$
\[
\xi_j=
\begin{cases}
x_jx_k^{-\frac{a_i}{a_k}} &\text{if $j\neq k$,}\\[6pt]
x_k^{\frac{r}{a_k}} &\text{if $j= k$.}\\
\end{cases}
\]
Note that $M_k$ contains the lattice $M$ and that the lattice inclusion $L_{k} \hookrightarrow L$ induces a toric morphism 
\begin{equation*}
\varphi \colon \Spec \C[\sigma_{k}\dual \cap M_{k}] \rightarrow U_{k}:=\Spec \C[\sigma_{k}\dual \cap M].
\end{equation*}
With eigencoordinates $\{\xi_1,\ldots, \xi_n\}$, the toric affine variety $U_{k}$ has a quotient singularity of type 
\[
\frac{1}{a_k}(a_1,\ldots,\underbrace{-r}_{\mbox{$k$th}},\ldots,a_n).
\]

\begin{Eg} \label{Eg:Star subdivision of 1/20(1,3,4)} Consider the group $G$ of type $\frac{1}{20}(1,3,4)$. Figure~\ref{Fig:Star subdivision for the type 1/20(1,3,4)} shows the star subdivision at $v=\frac{1}{20}(1,3,4)$.

\begin{figure}[h]
\begin{center}
\begin{tikzpicture}
\coordinate [label=left:$e_3$] (e3) at (0,0);
\coordinate [label=right:$e_2$] (e2) at (8,0);
\coordinate [label=right:$e_1$] (e1) at (4,6.8);
\coordinate (v1) at (4,1);
\draw[fill] (v1) circle [radius=0.05];
\draw[fill] (e3) circle [radius=0.05];
\draw[fill] (e2) circle [radius=0.05];
\draw[fill] (e1) circle [radius=0.05];
\node [above right] at (v1) {$v=\frac{1}{20}(1,3,4)$};
\draw (e3) -- (e2);
\draw (e3) -- (e1);
\draw (e2) -- (e1);
\draw (e3) -- (v1);
\draw (e2) -- (v1);
\draw (v1) -- (e1);
\node [below] at (4.2,0.6) {$\sigma_{1}$};
\node  at (2.8,3) {$\sigma_{2}$};
\node  at (5.3,3) {$\sigma_{3}$};
\coordinate [label=left:$v_7$] (v7) at (2.5,1.9);
\draw[fill] (v7) circle [radius=0.05];
\coordinate [label=right:$v_5$] (v5) at (6.92,1.87);
\draw[fill] (v5) circle [radius=0.05];
\coordinate [label=right:$v_{10}$] (v10) at (6,3.4);
\draw[fill] (v10) circle [radius=0.05];
\coordinate [label=right:$v_{15}$] (v15) at (5,5.1);
\draw[fill] (v15) circle [radius=0.05];
\end{tikzpicture}
\end{center}
\caption{Star subdivision at $v$ for the type $\frac{1}{20}(1,3,4)$}\label{Fig:Star subdivision for the type 1/20(1,3,4)}
\end{figure}
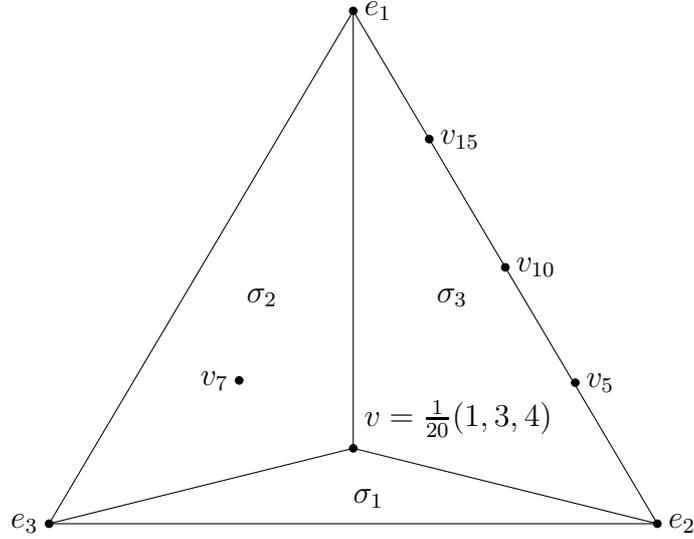

The cone $\sigma_2$ corresponds to the quotient singularity of type $\frac{1}{3}(1,1,1)$ with eigencoordinates $x y^{-\frac{1}{3}}, y^{\frac{20}{3}}, y^{-\frac{4}{3}}z$. There exists a unique lattice point $v_7=\frac{1}{20}(7,1,8)$ on the plane containing $e_1,v,e_3$. On the other hand, the cone $\sigma_3$ on the right side of $v$ has a singularity of type $\frac{1}{4}(1,3,0)$. Note that there are three other lattice points $v_5,v_{10},v_{15}$ on the plane containing $e_1,e_2,v$. Lastly, the affine toric variety corresponding to the cone $\sigma_1=\Cone(e_2,e_3,v)$ is smooth as $v,e_2,e_3$ form a $\Z$-basis of $L$.\eeg

\begin{Def}[Round down functions\footnote{As is stated in~\cite{Thesis,Terminal}, Davis, Logvinenko, and Reid\cite{DLR} introduced a similar construction in a more general setting.}]
With the notation above, for $k\in\{1,\ldots,n\}$, the {\em $k$-th round down function} $\phi_{k} \colon \lau \to  M_{k}$ of the star subdivision at $\frac{1}{r}(a_1,\ldots,a_n)$ is defined by
\[
\phi_{k}(x_1^{m_1}\cdots x_n^{m_n})=\xi_1^{m_1}\cdots\xi_k^{\lf \frac{1}{r}\sum a_im_i \rf}\ldots\xi_n^{m_n}.
\]
where $\lfloor \ \rfloor $ is the floor function. 
\end{Def}
Observe that since $v\in L$, $\frac{1}{r}\sum a_im_i$ is an integer if and only if the monomial $x_1^{m_1}\cdots x_n^{m_n}$ is $G$-invariant.

For the star subdivision at $\frac{1}{r}(a_1,\ldots,a_n)$, let $G_k$ denote the abelian group $L/L_k$ for $k\in\{1,\ldots,n\}$.
\begin{Lem}\label{Lem:phi(m+n)=phi(m)+(n)/phi_k induces a surjective map G dual -> G_k dual}
For each $k$, let $\phi_k$ be the round down function of the star subdivision at $\frac{1}{r}(a_1,\ldots,a_n)$ and $G_k$ the abelian group $L/L_k$. For a $G$-invariant monomial $\bn \in M$ and a monomial $\bm \in \lau$, 
\begin{equation*}
\phi_{k}(\bm \cdot \bn)= \phi_{k}(\bm)\cdot\bn.
\end{equation*}
In particular, the weights of $\phi_{k}(\bm \cdot \bn)$ and $\phi_{k}(\bm)$ are the same with respect to the $G_k$-action.
Thus $\phi_k$ induces a well-defined surjective map
\[
\phi_k \colon G\dual \rightarrow G_k\dual,\quad \rho \mapsto \phi_k(\rho),
\]
where $\phi_k(\rho)$ is the weight of $\phi_k(\bm)$ for a monomial $\bm \in \lau$ of weight~$\rho$. 
\end{Lem}
\begin{Rem}\label{Rem:description of G dual -> G_k dual}
In the lemma above, $\phi_k \colon G\dual \rightarrow G_k\dual$ can be described as follows. Let $\rho_i$ be the irreducible representation of $G$ whose weight is $i$. Then the weight of the representation $\phi_k(\rho_i)$ is $j$ where $j$ is the residue of $i$ modulo $a_k$, i.e. $j= i\mod a_k$.
\erem
\begin{Lem}\label{Lem:localizations,lacing}
Let $\bm \in \lau$ be a monomial of weight~$j$.
The  weight $j$ satisfies $0\leq j<r-a_k$ if and only if
\[
\phi_{k}(x_k\cdot \bm)=\phi_{k}(\bm).
\]
\end{Lem}
\begin{proof}
Assume that $\bm=x_1^{m_1}\cdots x_n^{m_n}$ is a monomial of weight $j$ with $0\leq j<r-a_k$, i.e.\ 
\[
0 \leq \frac{1}{r}\sum a_im_i - \lf \frac{1}{r}\sum a_im_i \rf <\frac{r-a_k}{r}.
\]
This is equivalent to the condition that $\phi_{k}(x_k\cdot \bm)=\phi_{k}(\bm)$.
\end{proof}
\begin{Lem}\label{Lem:bm=bn dot bm'}
If $\phi_k(\bm)=\phi_k(\bm')$ for some $k$, then $\bm=\bn\cdot \bm'$ or $\bm'=\bn\cdot \bm$ for some $\bn\in\mon$.
\end{Lem}
\begin{proof}
Let us suppose that $\phi_k(\bm)=\phi_k(\bm')$ for $\bm=x_1^{m_1}\cdots x_n^{m_n}$ and $\bm'=x_1^{m'_1}\cdots x_n^{m'_n}$ with $m_k\geq m'_k$. From the definition of the round down function $\phi_k$, we have $m_i=m'_i$ for all $i\neq k$. Thus $\bm=\bn\cdot \bm'$ with $\bn=x_k^{m_k-m'_k} \in \mon$.
\end{proof}
\begin{Def}
The star subdivision of $\sigma_+$ at $v=\frac{1}{r}(a_1,\ldots,a_n)$ is said to be {\em good} if:
\begin{enumerate}
\item $v$ generates $L/\Z^n$; 
\item for every $i\neq j$, $a_i+a_j \leq r$.
\end{enumerate}
\end{Def}
\begin{Lem}\label{Lem:Connected Gamma}
Let $\phi_k$ be the $k$-th round down function of the good star subdivision at $\frac{1}{r}(a_1,\ldots,a_n)$.
For any monomial $\bk$ in the lattice $M_k$ and any degree one monomial $\xi_j$ in $M_k$, there exist $x_i$ and a monomial $\bm \in \lau$ such that
\[
\phi_k(x_i \cdot \bm)=\xi_j\cdot\bk \quad \text{with} \quad \phi_k(\bm)= \bk.
\] 
\end{Lem}
\begin{proof}
Fix $k$. Suppose that $\bk$ is a monomial in $M_k$ and that $\xi_j$ is a degree one monomial in $M_k$.

First consider the case where $j=k$, i.e. $\xi_j=\xi_k$. As the round down function $\phi_k$ is surjective, there exists $\bm=x_1^{m_1}\cdots x_n^{m_n} \in \lau$ such that $\phi_k(\bm)=\bk$. By Lemma~\ref{Lem:localizations,lacing}, after multiplying $x_k$ enough, we may assume that $\phi_k(x_k\cdot\bm)\neq\phi_k(\bm)$. This means that
\[
\frac{1}{r}\sum_i a_im_i + \frac{a_k}{r} \geq \lf \frac{1}{r}\sum_i a_im_i \rf +1.
\]
Thus we have $\phi_k(x_k\cdot\bm)=\xi_k\cdot \bk$.

For the case where $j\neq k$, consider $\bm=x_1^{m_1}\cdots x_n^{m_n} \in \lau$ such that $\phi_k(\bm)=\bk$ with $\phi_k(x_k\inv\cdot\bm)\neq\phi_k(\bm)$, i.e.\
\[
\frac{1}{r}\sum_i a_im_i - \frac{a_k}{r} < \lf \frac{1}{r}\sum_i a_im_i \rf.
\]
Since the star subdivision is good, we have $a_k+a_j \leq r$. 
This implies that $\phi_k(x_j\cdot\bm)=\xi_j\cdot \bk$.
\end{proof}
\begin{Prop} \label{Prop:From Gamma' to Gamma:S(Gamma)=S(Gamma')}
Let $\phi_k$ be the $k$-th round down function of the good star subdivision at $\frac{1}{r}(a_1,\ldots,a_n)$. For a $G_k$-brick $\Gamma'$, define
\[
\Gamma:=\left\{\bm \in \lau \st \phi_{k}(\bm) \in \Gamma' \right\}.
\]
\begin{enumerate}
\item The set $\Gamma$ is a $G$-brick with $S(\Gamma)=S(\Gamma')$.
\item For $\bm \in \lau$, we have $\wtgp\big(\phi_k(\bm)\big)=\phi_k\big(\wtga(\bm)\big)$.
\end{enumerate} 
\end{Prop}
\begin{proof}
First we show (ii) assuming that $\Gamma$ is a $G$-prebrick. It follows that $\phi_k(\bm)$ is of the same weight as $\phi_k\big(\wtga(\bm)\big)$ from Lemma~\ref{Lem:phi(m+n)=phi(m)+(n)/phi_k induces a surjective map G dual -> G_k dual}. Since $\phi_k\big(\wtga(\bm)\big)\in\Gamma'$, the assertion is proved.

To prove (i), note that $\one \in \Gamma$ as $\phi_{k}(\one)=\one \in \Gamma'$. Second we show that there exists a unique monomial of weight $\rho$ in $\Gamma$ for each $\rho \in G\dual$. Fix $\rho\in G\dual$. Since the star subdivision is good, we have a monomial $\bm\in \lau$ such that the weight of $\bm$ is $\rho$. Note that $\wtgp\big(\phi_k(\bm)\big) \in \Gamma'$ and that $\frac{\wtgp\big(\phi_k(\bm)\big)}{\phi_k(\bm)}$ is in the lattice $M$. From Lemma~\ref{Lem:phi(m+n)=phi(m)+(n)/phi_k induces a surjective map G dual -> G_k dual}, 
\[
\phi_{k}\colon \bm \cdot \Big(\frac{\wtgp\big(\phi_k(\bm)\big)}{\phi_k(\bm)} \Big) \mapsto \wtgp\big(\phi_k(\bm)\big),
\]
so $\bm \cdot \Big(\frac{\wtgp\big(\phi_k(\bm)\big)}{\phi_k(\bm)} \Big)$ is an element of weight $\rho$ in $\Gamma$. From Lemma~\ref{Lem:phi(m+n)=phi(m)+(n)/phi_k induces a surjective map G dual -> G_k dual}, the uniqueness is followed. Lemma~\ref{Lem:Connected Gamma} implies that $\Gamma$ is connected as $\Gamma'$ is connected.

To show (iii) in Definition~\ref{Def:G-prebrick}, suppose that $\bn' \cdot \bn \cdot \bm_{\rho} \in \Gamma$ for $\bm_{\rho}\in \Gamma$ and $\bn, \bn' \in \mon$. Note that 
\[
\phi_{k}(\bn' \cdot \bn \cdot \bm_{\rho})=
\frac{\phi_{k}(\bn' \cdot \bn \cdot \bm_{\rho})}{\phi_{k}(\bn \cdot \bm_{\rho})} 
\cdot 
\frac{\phi_{k}(\bn \cdot \bm_{\rho})}{\phi_{k}(\bm_{\rho})}
\cdot 
{\phi_{k}(\bm_{\rho})}
\in \Gamma'.
\]
Since $\Gamma'$ is a $G_k$-brick, $\phi_{k}(\bn \cdot \bm_{\rho})\in\Gamma'$. Thus $\bn \cdot \bm_{\rho}$ is in $\Gamma$. Therefore $\Gamma$ is a $G$-prebrick. 

To show that $S(\Gamma)=S(\Gamma')$, note that for $\bn \in \mon$ and $\bm_{\rho} \in \Gamma$,
\[
\frac{\bn \cdot \bm_{\rho}}{\wtga(\bn \cdot \bm_{\rho})} 
= \frac{\phi_{k}(\bn \cdot \bm_{\rho})}{\phi_{k}\big(\wtga(\bn \cdot \bm_{\rho})\big)}
= \frac{\bk \cdot \phi_{k}(\bm_{\rho})}{\wtgp\big(\bk \cdot \phi_{k}(\bm_{\rho}) \big)} \in S(\Gamma')
\]
where $\bk=\frac{\phi_{k}(\bn\cdot \bm_{\rho})}{\phi_{k}(\bm_{\rho})}$. Since $S(\Gamma)$ is generated by $\frac{\bn \cdot \bm_{\rho}}{\wtga(\bn \cdot \bm_{\rho})}$, we proved that 
$S(\Gamma) \subset S(\Gamma')$.

For the opposite inclusion, suppose that $\bk\in\Gamma'$. Let $\{\xi_j\}$ be the eigencoordinates with respect to the $G_k$-action.
Lemma~\ref{Lem:Connected Gamma} shows that for every $\xi_j$ there exist $x_i$, $\bm_{\rho} \in \Gamma$ such that $\phi_{k}(x_i \cdot \bm_{\rho})=\xi_j\cdot\bk$ with $\phi_{k}(\bm_{\rho})=\bk$. Then
\[
\frac{\xi_j \cdot \bk}{\wtgp(\xi_j \cdot \bk)} 
=\frac{\phi_{k}(x_i \cdot \bm_{\rho})}{\wtgp\big(\phi_{k}(x_i \cdot \bm_{\rho})\big)}
= \frac{\phi_{k}(x_i \cdot \bm_{\rho})}{\phi_{k}\big(\wtga(x_i \cdot \bm_{\rho})\big)}
=\frac{x_i \cdot \bm_{\rho}}{\wtga(x_i \cdot \bm_{\rho})}.
\]
This completes the proof.
\end{proof}
\begin{Def}
The $G$-brick $\Gamma$ in Proposition~\ref{Prop:From Gamma' to Gamma:S(Gamma)=S(Gamma')} is called the \emph{natural inverse of $\Gamma'$} and denoted by $\roundbrick(\Gamma')$.
\end{Def}
\subsection{Star subdivisions and bricksets}\label{Sec:star subdivisions and bricksets}
Let $G$ be a finite diagonal group in $\GL_n(\C)$. Let $X$ denote the quotient variety $\C^n/G$ and $X_v$ the toric variety given by the good star subdivision at $v=\frac{1}{r}(a_1,\ldots,a_n)$. Recall that the toric fan of $X_v$ contains the $n$-dimensional cone
\[
\sigma_k:=\Cone(e_1,\ldots,\hat{e_k},v,\ldots,e_n).
\]
Note that $X_v$ is covered by the affine toric open sets
\[
U_k=\Spec\C[\sigma_k\dual\cap M]\iso \C^n/G_k,
\]
where $G_k=L/L_k$.

Assume that $Y$ is a normal toric variety admitting a proper birational morphism $Y\to X$. Let $\Sigma$ denote the toric fan of $Y$. Assume further that there exists a dominant toric morphism $\overline{\varphi}\colon Y\to X_v$ fitting into the commutative diagram:
\[ 
\begin{tikzcd} 
Y \arrow{r}{\overline{\varphi}}\arrow{rd} & X_v\arrow{d} \\&X.\end{tikzcd}\]

As is standard in toric geometry (see eg.\ Section~3.3~in~\cite{CLS}), for each cone $\sigma\in\Sigma$, there exists a cone $\sigma_k$ such that $\sigma \subset \sigma_k$. Therefore for each $k$, $\varphi$ induces the following toric morphism
\[
\overline{\varphi}_k\colon Y_k \to U_k,
\]
where $Y_k$ is the toric variety whose fan consists of the cones $\sigma \in \Sigma$ satisfying $\sigma \subset \sigma_k$. Note that since $X_v\to X$ is projective, if the morphism $Y\to X$ is projective, then so is $\varphi_k$.
\begin{Thm}\label{Thm:subdivision and brickset}
With the assumption above, further assume that each $Y_k\to U_k$ has a $G_k$-brickset $\gr_k$. Define
\[
\gr:=\bigcup_k  \left\{\roundbrick(\Gamma') \st \Gamma' \in \gr_k \right\}.
\]
Then $\gr$ is a $G$-brickset for the morphism $Y\to X$.
\end{Thm}
\begin{proof} 
Let $\mco$ be the set of the $n$-dimensional cones in the fan $\Sigma$ of $Y$.
From Proposition~\ref{Prop:From Gamma' to Gamma:S(Gamma)=S(Gamma')}, it follows that every object in the set $\gr$ is a $G$-brick. It suffices to prove that the set $\gr$ satisfies:
\begin{enumerate}
\item there exists a bijection $\mco \to \gr$ sending $\sigma$ to $\Gamma_{\sigma}$;
\item $S(\Gamma_{\sigma})=\sigma\dual \cap M$. 
\end{enumerate}

Let $\sigma$ be an arbitrary $n$-dimensional cone in $\Sigma$. Then there exists a unique cone $\sigma_k$ such that $\sigma \subset \sigma_k$. By the assumption, there is a unique $G_k$-brick $\Gamma' \in \gr_k$ such that  
$S(\Gamma')=\sigma \dual \cap M$.
Define
\[
\Gamma_{\sigma}=\roundbrick(\Gamma'):=\left\{\bm \in \lau \st \phi_{k}(\bm) \in \Gamma' \right\}.
\]
By Proposition~\ref{Prop:From Gamma' to Gamma:S(Gamma)=S(Gamma')}, we have 
\[
S(\Gamma_{\sigma})=S(\Gamma')=\sigma \dual \cap M,
\]
and the proof is completed.
\end{proof}
Please note that by Proposition~\ref{Prop:if we have a brickset enough to show there exists theta}, if there is $\theta\in\Theta$ satisfying that every $\Gamma$ in $\gr$ is $\theta$-stable, then $Y$ is isomorphic to $\yth$.
\subsection{Star subdivisions and stability parameters}\label{Sec:star subdivisions and stability parameters}
In this section, we discuss the existence of a stability parameter $\theta$ such that every $G$-brick in the brickset described in Theorem~\ref{Thm:subdivision and brickset} is $\theta$-stable.

Consider the good star subdivision at $v=\frac{1}{r}(a_1,\ldots,a_n)$. 
For each $k$, let $\Theta^{(k)}$ be the GIT parameter space of $G_k$-constellations.
Remember that by Lemma~\ref{Lem:phi(m+n)=phi(m)+(n)/phi_k induces a surjective map G dual -> G_k dual} we have the well-defined surjective map
\[
\phi_k \colon G\dual \to G_k\dual
\]
induced by the round down function $\phi_k$. Note that the linear map
\[
\roundstab\colon \Theta \to \Theta^{(k)}
\]
defined by
\begin{equation}
[\roundstab(\theta)](\chi)=\sum\limits_{\phi_k(\rho)=\chi} \theta(\rho) \text{ for $\chi \in G_k\dual$}
\end{equation}
is well-defined.

Let $\rho_i$ denote the irreducible representation of $G$ whose weight is $i$ and $\chi_j$ the irreducible representation of $G_k$ whose weight is $j$.
First note that $\Theta$, which is a $\Q$-vector space of $(r-1)$-dimension, has a $\Q$-basis $\left\{\theta_i\in \Hom_{\Z} (R(G),\Q)\st 1 \leq i < r\right\}$ where
\begin{equation}\label{Eqtn:basis of Theta}
\theta_i(\rho_l)=
\begin{cases}
1 &\text{if $\rho_l=\rho_i$,}\\
-1 &\text{if $\rho_l$ is trivial,}\\
0 &\text{otherwise,}\\
\end{cases}
\end{equation}
for $\rho_l\in G\dual$. By the definition of the round down functions (see Lemma~\ref{Lem:phi(m+n)=phi(m)+(n)/phi_k induces a surjective map G dual -> G_k dual} or Remark~\ref{Rem:description of G dual -> G_k dual}), we have
\[
[\roundstab(\theta_i)](\chi_j)=
\begin{cases}
1 &\text{if $j = i \mod a_k$,}\\
-1 &\text{if $\chi_j$ is trivial,}\\
0 &\text{otherwise,}\\
\end{cases}
\]
for $\chi_j\in G_k\dual$. In particular, $\roundstab(\theta_i) \equiv \roundstab(\theta_{i'})$ if $i \equiv i' \mod a_k$.
\begin{Rem}
As is discussed above, $\roundstab\colon \Theta \to \Theta^{(k)}$ is surjective. Indeed, $\roundstab(\theta_i)$ for $1 \leq i<a$ form a $\Q$-basis of $\Theta^{(k)}$.
\erem

From now on, we only consider 3-dimensional cases. Consider the good star subdivision at $v=\frac{1}{r}(1,a,b)$ with $a < b$.
\begin{Lem}\label{Lem:the existence of partial solns for star subdivision}
Consider the good star subdivision at $v=\frac{1}{r}(1,a,b)$ with $a<b$. Assume that $a$ and $b$ are coprime. Given $\theta^{(k)}\in\Theta^{(k)}$ for $k=2,3$, there exists $\theta_P \in \Theta$ such that
\begin{equation}\label{Eqtn:partial solution}
\roundstab(\theta) \equiv \theta^{(k)}
\end{equation}
for all $k$.
\end{Lem}
\begin{proof}
Consider the linear map
\[
\phi_{\star}=\left((\phi_2)_{\star},(\phi_3)_{\star}\right)\colon \Theta \to \Theta^{(2)}\oplus \Theta^{(3)}.
\]
We need to prove that $\phi_{\star}$ is surjective. It suffices to show that \[
\left\{ \phi_{\star}(\theta_i) \st 1 \leq i \leq a+b-2 \right\}
\]
is linearly independent as the dimension of $\Theta^{(2)}\oplus \Theta^{(3)}$ is $a+b-2$. 
Using the fact that $a$ and $b$ are coprime, the assertion follows from a direct calculation.
\end{proof}
\section{Main theorem}
\subsection{Toric minimal model program}\label{Sec:toric mmp} In this section, we recall the birational geometry of toric varieties (see \cite{Rtoric}).
Reid\cite{Rtoric} introduced a combinatorial criterion for a toric variety to have terminal singularities and canonical singularities.
\begin{Thm}[Reid\cite{Rtoric}]\label{Thm:Criterion of term.can sings} Let $X$ be the toric variety corresponding to a fan $\Sigma$ with  a lattice $L$ and the dual lattice $M$. Then $X$ has only terminal singularities (resp.\ canonical singularities) if and only if any cone $\sigma \in \Sigma$ satisfies the conditions (i) and (ii) (resp.\ (i) and (iii)):
\begin{enumerate}
\item there exists an element $\bm \in M_{\Q}$ such that $\langle \bu,\bm \rangle=1$ for any primitive vector $\bu$ of $\sigma$;
\item there are no other lattice points in the set $\left\{ \bu \in \sigma \st \langle \bu, \bm \rangle \leq 1 \right\}$ except vertices;
\item there are no other lattice points in the set $\left\{ \bu \in \sigma \st \langle  \bu, \bm \rangle < 1 \right\}$ except the origin.
\end{enumerate}
\end{Thm}
\begin{Thm}[Reid\cite{Rtoric}]\label{Thm:Toric MMP}
Let $X$ be a quasiprojective toric variety and $V\rightarrow X$ a projective birational toric morphism with $V$ smooth. Then there exists the following diagram
\[ 
\begin{tikzcd} 
V \arrow[dashed]{r}\arrow{rd} & Y \arrow{r}{\overline{\varphi}}\arrow{d}{{\varphi}} & X_{\mathrm{can}}\arrow{ld}{\nu} \\&X,\end{tikzcd}\]
where
\begin{enumerate}
\item $X_{\mathrm{can}}$ has canonical singularities, ${\nu} \colon X_{\mathrm{can}} \to X$ is a projective birational morphism, and $K_{X_{\mathrm{can}}}$ is $\nu$-ample;
\item $Y$ has $\Q$-factorial terminal singularities, ${\varphi} \colon Y \to X$ is a projective birational morphism, and $K_{Y}$ is ${\varphi}$-nef, i.e. $\overline{\varphi}$ is crepant.
\end{enumerate}
\end{Thm}
\begin{Def}
In Theorem \ref{Thm:Toric MMP}, we say that:
\begin{enumerate}
\item the variety $X_{\mathrm{can}}$ is a {\em relative canonical model} of $X$;
\item the variety $Y$ is a {\em relative minimal model} of $X$.
\end{enumerate}
\end{Def}
\begin{Convention}
In this article, relative minimal models of $X$ are always projective over $X$.
\end{Convention}
\subsection{The Craw--Ishii conjecture} \label{Sec:main section}
For a finite abelian subgroup $G$ of $\SL_3(\C)$, Craw and Ishii proved that every projective crepant resolution of $\C^3/G$ is isomorphic to $\mth$ for a suitable parameter $\theta$. 
\begin{Thm}[Craw--Ishii\cite{CI}]\label{Thm:CI theorem}
For a finite abelian subgroup $G$ of $\SL_3(\C)$, let $Y$ be a relative minimal model of $\C^3/G$. Then $Y$ is isomorphic to $\mth$ for a suitable $\theta$.
\end{Thm}

They conjectured that the same holds without the abelian assumption. We further conjecture that the same is true for all finite group $G\subset \GL_3(\C)$ if $Y$ is a smooth relative minimal model.
\begin{Con}[Craw--Ishii conjecture]\label{Craw--Ishii conjecture}
For a finite subgroup $G$ of $\GL_3(\C)$, let $Y$ be a relative minimal model of $\C^3/G$. If $Y$ is smooth, then $Y$ is isomorphic to (the birational component $\yth$ of) $\mth$ for a suitable $\theta$.
\end{Con}

To prove the theorem above, Craw and Ishii showed that a flop of $\mth$ is isomorphic to $\sM_{\theta'}$ for some pamameter $\theta'$ as two crepant resolutions are connected by a sequence of flops. This completes the proof because we already knew that $\GHilb{3}$ is a crepant resolution of $\C^3/G$ by Bridgeland--King--Reid~\cite{BKR} for $G\subset \SL_3(\C)$. 

Note that for $G \not\subset \SL_3(\C)$ we do not have a moduli description of any relative minimal model of $\C^3/G$ yet.

\begin{Rem}
From Theorem~\ref{Thm:CI theorem}, we have a simple corollary as follows. For $G$ a finite abelian subgroup in $\SL_3(\C)$ and $Y$ a projective crepant resolution of $X=\C^3/G$, there exist:
\begin{enumerate}
\item a $G$-brickset $\gr$ for $Y\to X$;
\item a stability parameter $\theta$ such that $\Gamma \in \gr$ is $\theta$-stable.
\end{enumerate}
We use this to prove the Craw--Ishii conjecture for some cases.\erem

\subsection{\for{toc}{The first case: $r=abc+a+b+1$}\except{toc}{The first case: \boldmath${r=abc+a+b+1}$}}\label{Sec:1st case} In this section, as the first example, we prove the Craw--Ishii conjecture for the group $G$ of type $\frac{1}{r}(1,a,b)$ with $r=abc+a+b+1$ where $a,b,c$ are positive integers with $b$ coprime to $a$.
Consider the lattice
\[
L = \Z^3 + \Z\cdot \frac{1}{r}(1,a,b),
\]
and the cone $\sigma_+=\Cone(e_1,e_2,e_3)$. The lattice point $v:=\frac{1}{r}(1,a,b)$ is in the interior of the simplex $\bigtriangleup$ where
\[
\bigtriangleup:=\{\bu \in \sigma_+ \st \langle \bu,x_1x_2x_3 \rangle \leq 1 \},
\]
with considering the monomial $x_1x_2x_3$ as an element in $M_{\Q}$.
Thus the quotient singularity $X:=\C^3/G$ defined by the cone $\sigma_+$ is not a canonical singularity.

Consider the star subdivision of $\sigma_+$ at $v$. Let $X_v$ denote the toric variety corresponding to the star subdivision. From Section~\ref{Sec:star subdivisions and round down functions}, we have:
\begin{enumerate}
\item the cone $\sigma_{1}=\Cone(v,e_2,e_3)$ defines a smooth open set $U_1$;
\item the cone $\sigma_{2}=\Cone(e_1,v,e_3)$ defines the affine toric variety $U_2=\C^3/G_2$ with $G_2$ of type $\frac{1}{a}(1,-r,b)$;
\item the cone $\sigma_{3}=\Cone(e_1,e_2,v)$ defines the affine toric variety $U_3=\C^3/G_3$ with $G_3$ of type $\frac{1}{b}(1,a,-r)$.
\end{enumerate}
Note that $\sigma_{2}$ and $\sigma_{3}$ define Gorenstein 3-fold abelian quotient singularities. Hence the star subdivision at $v$ has only canonical singularities. Since a star subdivision induces a projective toric morphism, from the ramification formula (\ref{Eqtn:ramification formula for star subdivision}), it follows that the star subdivision of $\sigma_+$ at $v$ defines the relative canonical model $X_{\mathrm{can}}$ of $X$, i.e. $X_v$ is the relative canonical model of $X$.

Suppose that $\varphi\colon Y\to X$ is a relative minimal model. Then there exists a projective crepant morphism $\overline{\varphi}\colon Y \to X_v$ fitting into the following commutative diagram:
\[ 
\begin{tikzcd} 
Y \arrow{r}{\overline{\varphi}}\arrow{rd}{{\varphi}} & X_v\arrow{d} \\& X.\end{tikzcd}\]
As is discussed in Section~\ref{Sec:star subdivisions and bricksets}, we have the following three induced projective crepant morphisms:
\begin{enumerate}
\item $\overline{\varphi}_1\colon Y_1 \to U_1= \C^3$;
\item $\overline{\varphi}_2\colon Y_2 \to U_2=\C^3/G_2$;
\item $\overline{\varphi}_3\colon Y_3 \to U_3=\C^3/G_3$.
\end{enumerate}
Here $Y_k$ denotes the toric variety given by the cones $\sigma$ such that $\sigma \subset \sigma_k$. Note that $\overline{\varphi}_1$ is an isomorphism. From the Craw--Ishii Theorem~\cite{CI}, it follows that there exists a generic GIT parameter $\theta^{(k)}$ such that $Y_k$ is the moduli space of $\theta^{(k)}$-stable $G_k$-constellations. Thus we have:
\begin{enumerate}
\item a $G_k$-brickset $\gr_k$ for $Y_k\to U_k$;
\item a stability parameter $\theta^{(k)}$ such that $\Gamma \in \gr_k$ is $\theta^{(k)}$-stable.
\end{enumerate}

By Theorem~\ref{Thm:subdivision and brickset}, there exists a $G$-brickset $\gr$ for $Y\to X$. Note that to the cone $\sigma_1$, we assign the $G$-brick
\[
\Gamma:=\phi_{1}\inv(\one)=\{1,x_1,x_1^2,\ldots,x_1^{r-1}\},
\]
which satisfies $S(\Gamma)=\sigma_1\dual\cap M$.

Now we show that there exists a parameter $\theta$ such that every $\Gamma \in \gr$ is $\theta$-stable. First note that since $a$ and $b$ are coprime, by Lemma~\ref{Lem:the existence of partial solns for star subdivision}, there exists $\theta_P \in \Theta$ satisfies (\ref{Eqtn:partial solution}) for given $\theta^{(2)}$ and $\theta^{(3)}$.

Define the GIT parameter $\vtheta\in\Theta$ by
\begin{equation}\label{Eqtn:def. of vartheta for the 1st case}
\vtheta(\rho) =
\begin{cases}
-1 &\text{if $0 \leq \wt(\rho)<b$,}\\
-1 &\text{if $\wt(\rho)=a+b$,}\\
1 &\text{if $r-b-1 \leq \wt(\rho)<r$,}\\
0 &\text{otherwise.}\\
\end{cases}
\end{equation}
\begin{Rem}\label{Rem:key properties of vartheta}
The parameter $\vtheta$ in~(\ref{Eqtn:def. of vartheta for the 1st case}) has the following properties for each $k$.
\begin{enumerate}
\item For any $\chi \in G_k\dual$, we have $[\roundstab(\vtheta)](\chi)\equiv 0$, i.e.
\[
\sum_{i = j \mod a_k} \vtheta(\rho_i) =0
\]
for each $0\leq j <a_k$.
\item For any $i$ with $0 \leq i <a_k$, we have $\vtheta(\rho_i) <0$.
\item For a monomial $\bm$ of weight $i$ with $a_k\leq i <r$, define 
\[
A:=A(\bm):=\left\{ x_k^l\ \cdot \bm \st \phi_k\big( x_k^l \cdot \bm\big)=\phi_k(\bm) \text{ for some $l\geq 0$}\right\}.
\]
Then we have $\vtheta(A) >0$.
\end{enumerate}
These are the key properties we use for the existence of a suitable parameter.
\erem
\begin{Prop}\label{Prop:our stability for 1st case}
For a sufficiently large natural number $m$, set
\begin{equation}\label{Eqtn:stability parameter m large for the 1st case}
\theta := \theta_P + m \vtheta,
\end{equation}
with $\theta_P$ satisfying (\ref{Eqtn:partial solution}).
If a $G$-brick $\Gamma$ is in $\gr$ described above, then $\Gamma$ is $\theta$-stable.
\end{Prop}
\begin{proof} 
Let $\Gamma$ be a $G$-brick in $\gr$ and $\sigma$ the corresponding cone. 

If the cone $\sigma$ is contained in $\sigma_1$, then the corresponding $G$-graph is $\Gamma=\{1,x_1,x_1^2,\ldots,x_1^{r-2},x_1^{r-1}\}$.
Note that any nonzero proper submodule $\sG$ of $C(\Gamma)$ is given by
\[
A=\{x_1^{j},x_1^{j+1}, \ldots,x_1^{r-2},x_1^{r-1}\}
\]
for some $1\leq j \leq r-1$ by Lemma~\ref{Lem:combinatorial description of submodule}. Thus $\vtheta(\sG)>0$ by definition. From this, it follows that $\Gamma$ is $\theta$-stable for sufficiently large $m$.

For the other cases, assume that $\Gamma$ is the $G$-brick corresponding to a cone $\sigma \subset \sigma_k$. Let $\Gamma'$ be the $G_k$-brick corresponding to $\Gamma$ and $\sG$ a nonzero proper submodule of $C(\Gamma)$ with $\C$-basis $A \subset \Gamma$. Recall that 
\[
\Gamma=\left\{\bm \in \lau \st \phi_{k}(\bm) \in \Gamma' \right\},
\]
as in Proposition~\ref{Prop:From Gamma' to Gamma:S(Gamma)=S(Gamma')}.
We have the following two cases:
\begin{enumerate}
\item $A  = \phi_k\inv\big(\phi_k(A)\big):=\left\{\bm \in \lau \st \phi_{k}(\bm) \in \phi_k(A) \right\}$;
\item $A \subsetneqq \phi_k\inv\big(\phi_k(A)\big)$.
\end{enumerate}

In case (i), $\vtheta(\sG)=0$ by definition. Moreover, we can see that the set $\phi_k(A)$ defines a nonzero proper submodule $\sG'$ of $C(\Gamma')$ as follows. Let $\xi_1,\xi_2,\xi_3$ be the the eigencoordinates with respect to the $G_k$-action on $\C^3$. Suppose that $\xi_j \cdot \phi_k(\bm_{\rho}) \in \Gamma'$ for some $\bm_{\rho} \in A$. 
Lemma~\ref{Lem:Connected Gamma} implies that there exist $\bm_{\rho'}$ and $x_i$ such that 
\[
\phi_k(x_i \cdot \bm_{\rho'})=\xi_j\cdot\phi_k(\bm_{\rho}) \quad \text{with} \quad \phi_k(\bm_{\rho'})=\phi_k(\bm_{\rho}).
\]
As $A$ is a $\C$-basis of $\sG$, Lemma~\ref{Lem:combinatorial description of submodule} implies that $x_i \cdot \bm_{\rho'}\in A$. Thus $\xi_j\cdot\phi_k(\bm_{\rho})$ is in $\phi_k(A)$.  This shows that $\phi_k(A)$ is a $\C$-basis of a nonzero proper submodule $\sG'$ of $C(\Gamma')$.
Since 
\[
\roundstab(\theta) \equiv \theta^{(k)},
\]
we have $\theta(\sG)=\theta^{(k)}(\sG')>0$ as $\sG'$ is a submodule of the $\theta^{(k)}$-stable constellation $ C(\Gamma')$.

Consider case (ii). Observe that 
\[
\sum\limits_{\phi_k(\rho')\in\phi_k(A)} \vtheta(\rho') = 0
\]
by the definition of $\vtheta$.
Lemma~\ref{Lem:localizations,lacing} implies that if $\bm_{\rho}$ in $\phi_k\inv\big(\phi_k(A)\big) \setminus A$, then $0 \leq \wt({\rho}) < r-b$. Moreover we have
\[
\sum\limits_{\rho'\in\phi_k\inv\left(\phi_k(A)\right)\setminus A} \vtheta(\rho') < 0.
\]
Thus $\vtheta(\sG)>0$. Therefore $\theta(\sG)>0$ for sufficiently large $m$.

Since there exist a finite number of $G$-bricks in $\gr$, we are done.
\end{proof}
\begin{Rem}\label{Rem:generic parameter for 1st case}
The parameter $\theta$ in Proposition~\ref{Prop:our stability for 1st case} does not need to be generic. However, since the condition for $\theta$ is an open condition, there exists a generic parameter in a small neighbourhood of $\theta$.
\erem
As we have proved the existence of a suitable generic parameter $\theta$, we have the following theorem.
\begin{Thm}\label{Thm:Main Theorem:1st case}
For positive integers $a,b,c$ with $b$ coprime to $a$, let $G$ be the group of type $\frac{1}{r}(1,a,b)$ with $r=abc+a+b+1$. Assume that $Y\rightarrow X:=\C^3/G$ is any relative minimal model of $X$. Then $Y$ is isomorphic to the birational component $\yth$ of the moduli space $\mth$ of $\theta$-stable $G$-constellations for a suitable parameter~$\theta$.
\end{Thm}
In some small cases (eg.\ $\frac{1}{20}(1,3,4)$), the irreducible component $\yth$ is actually a connected component of $\mth$. However, we do not have a big example such that $\mth$ itself is irreducible.
\begin{Que}
In the situation as above, is $\mth$ in the theorem above irreducible?
\end{Que}

\begin{Eg} \label{Eg:Calculating G-graphs in 1/20(1,3,4)} Let $G$ be the group of type $\frac{1}{20}(1,3,4)$ as in Example~\ref{Eg:Star subdivision of 1/20(1,3,4)}. Consider the star subdivision at $v=\frac{1}{20}(1,3,4)$. Then the star subdivision gives the relative canonical model of $X=\C^3/G$. 

Let $\phi \colon Y\to X$ be a relative minimal model whose fan is shown in Figure~\ref{Fig:recursion process of 1/20(1,3,4)}.

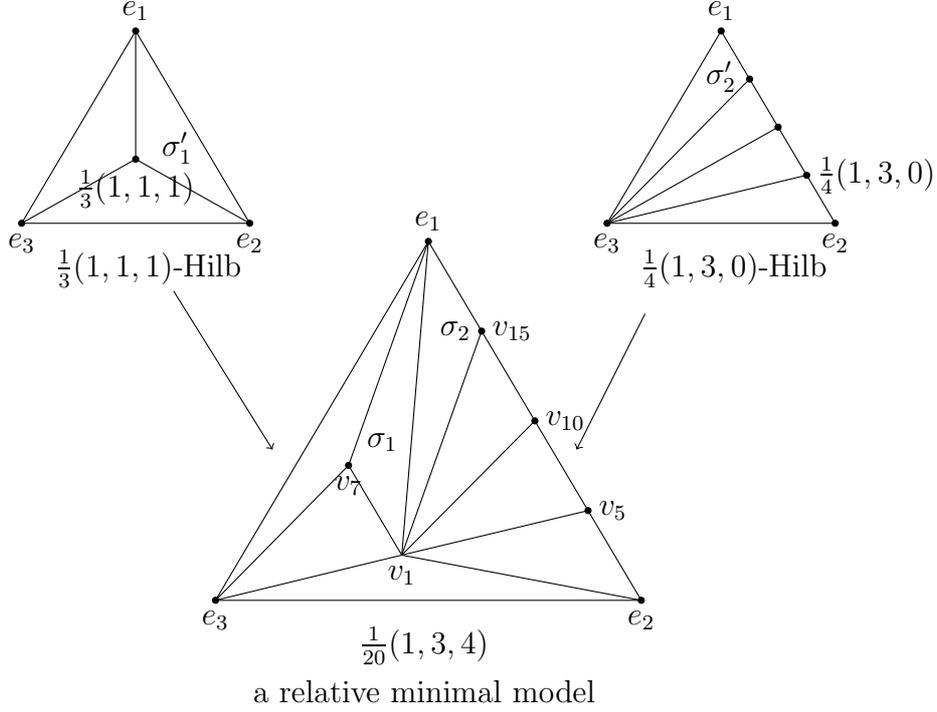
\begin{figure}[h]
\begin{center}
\begin{tikzpicture}
\coordinate [label=below:$e_3$] (re3) at (8.7-0.5-0.5,5);
\coordinate [label=below:$e_2$] (re2) at (11.7-0.5-0.5,5);
\coordinate [label=above:$e_1$] (re1) at (10.2-0.5-0.5,7.55);
\draw[fill] (re1) circle [radius=0.04];
\draw[fill] (re2) circle [radius=0.04];
\draw[fill] (re3) circle [radius=0.04];
\draw (re3) -- (re2);
\draw (re3) -- (re1);
\draw (re2) -- (re1);

\coordinate [label=below:$e_3$] (e3) at (3.05-0.5,0);
\coordinate [label=below:$e_2$] (e2) at (8.65-0.5,0);
\coordinate [label=above:$e_1$] (e1) at (5.85-0.5,4.76);
\draw[fill] (e1) circle [radius=0.04];
\draw[fill] (e2) circle [radius=0.04];
\draw[fill] (e3) circle [radius=0.04];
\draw (e3) -- (e2);
\draw (e3) -- (e1);
\draw (e2) -- (e1);

\coordinate [label=below:$e_3$] (le3) at (0,5);
\coordinate [label=below:$e_2$] (le2) at (3,5);
\coordinate [label=above:$e_1$] (le1) at (1.5,7.55);
\draw[fill] (le1) circle [radius=0.04];
\draw[fill] (le2) circle [radius=0.04];
\draw[fill] (le3) circle [radius=0.04];
\draw (le3) -- (le2);
\draw (le3) -- (le1);
\draw (le2) -- (le1);

\draw[->] (2,4.1) -- (3.8-0.5,2);
\draw[->] (9.2-0.5-0.5,4.1-0.3) -- (7.8-0.5,2);

\coordinate (lv) at (1.5,5.85);
\draw[fill] (lv) circle [radius=0.04];
\node [below] at (lv) {$\tfrac{1}{3}(1,1,1)$};
\foreach \x in {1,2,3}
\draw (le\x) -- (lv);
\node [right] at (1.7,6) {$\sigma'_{1}$};


\foreach \x in {1,2,3}
\coordinate (rv\x) at (11.7-0.5-0.5-\x*0.375, 5+\x*0.6375);
\foreach \x in {1,2,3}
\draw[fill] (rv\x) circle [radius=0.04];
\foreach \x in {1,2,3}
\draw (re3) -- (rv\x);
\node [right] at (rv1) {$\tfrac{1}{4}(1,3,0)$};
\node [below] at (10.2-0.5-0.5,7.3) {$\sigma'_{2}$};


\coordinate [label=below:$v_1$] (v1) at (5.5-0.5,0.6);
\foreach \x in {1,2,3}
\draw (v1) -- (e\x);

\foreach \x in {1,2,3}
\coordinate (vv\x) at (8.65-0.5-\x*0.7, \x*1.19);
\foreach \x in {1,2,3}
\draw[fill] (vv\x) circle [radius=0.04];
\foreach \x in {1,2,3}
\draw (v1) -- (vv\x);

\node [right] at (vv1) {$v_5$};
\node [right] at (vv2) {$v_{10}$};
\node [right] at (vv3) {$v_{15}$};

\coordinate [label=below:$v_7$] (v7) at (4.8-0.5,1.79);
\draw[fill] (v7) circle [radius=0.04];
\draw (v7) -- (e1);
\draw (v7) -- (v1);
\draw (v7) -- (e3);

\node [right] at (4.9-0.5,2.1) {$\sigma_{1}$};
\node [left] at (6.55-0.5,3.57) {$\sigma_{2}$};

    
\node [right] at (0.3,4.4) {$\tfrac{1}{3}(1,1,1)\HHilb$};
\node [right] at (4.8-0.5,-0.6) {$\tfrac{1}{20}(1,3,4)$};
\node [right] at (3.4-0.5,-1.2) {$\text{a relative minimal model}$};
\node [right] at (9.0-0.5-0.5,4.4) {$\tfrac{1}{4}(1,3,0)\HHilb$};

\end{tikzpicture}
\end{center}
\caption{Recursion process for $\frac{1}{20}(1,3,4)$}\label{Fig:recursion process of 1/20(1,3,4)}
\end{figure}
There exist the two induced projective crepant resolutions:
\begin{enumerate}
\item $\overline{\varphi}_2\colon Y_2 \to U_2=\C^3/G_2$;
\item $\overline{\varphi}_3\colon Y_3 \to U_3=\C^3/G_3$.
\end{enumerate}
Here $G_2$ is of type $\frac{1}{3}(1,1,1)$ and $G_3$ is of type $\frac{1}{4}(1,3,0)$. Note that $Y_2$ and $Y_3$ are $\GHii{2}$ and $\GHii{3}$, respectively.

We illustrate how to calculate $G$-bricks associated to the following cones:
\begin{align*}
\sigma_1 &:= \Cone\left((1,0,0), \tfrac{1}{20}(1,3,4),\tfrac{1}{20}(7,1,8)\right)\!, \\[3pt]
\sigma_2 &:= \Cone\left((1,0,0), \tfrac{1}{20}(1,3,4), \tfrac{1}{20}(15,5,0)\right)\!.
\end{align*}

Note that the cone $\sigma_1$ is in $\Cone(e_1,v_1,e_3)$. Moreover, observe that the left fan corresponds to $\GHii{2}$ with $G_2$ of type $\frac{1}{3}(1,1,1)$. Consider the cone $\sigma'_1$ in the fan of $\GHii{2}$ corresponding to $\sigma_1$. Let $\xi,\eta,\zeta$ denote the eigencoordinates for $G_2$. The corresponding $G_2$-brick is
\[
\Gamma'_1 = \big\{ 1, \zeta,\zeta^2 \big\}\!.
\]
The $G$-brick $\Gamma_1$ corresponding to $\sigma_1$ is
\[
\Gamma_1 \eqd \big\{ x^{m_1}y^{m_2}z^{m_3} \in \lau \st \phi_{2}(x^{m_1}y^{m_2}z^{m_3}) \in \Gamma'_1 \big\}
\]
where the left round down function $\phi_2$ is defined by
\[
\phi_{2}(x^{m_1}y^{m_2}z^{m_3})=\xi^{m_1}\eta^{\lf\frac{1}{20}m_1 + \frac{3}{20} m_2 + \frac{4}{20}m_3\rf}\zeta^{m_3}.
\]
Thus 
\[
\Gamma_1=\left\{
\begin{matrix}
y^{-2}z^2 & y^{-1}z^2& z^2 & yz^2 &y^2z^2&y^3z^2 \\
& y^{-1}z& z & yz &y^2z&y^3z&y^4z&y^5z \\
&&1 & y &y^2&y^3&y^4&y^5&y^6
\end{matrix}
\right\}\!.
\]

Observe that the cone $\sigma_2$ is in $\Cone(e_1,e_2,v_1)$.
The right fan is the fan of $\GHii{3}$, where $G_3$ is of type $\frac{1}{4}(1,3,0)$. Let $\alpha,\beta,\gamma$ be the eigencoordinates. For the cone $\sigma'_2$ corresponding to $\sigma_2$, observe that the corresponding $G_3$-brick is
\[
\Gamma'_2 = \big\{ 1, \beta,\beta^2, \beta^3 \big\}.
\]
The $G$-brick $\Gamma_2$ corresponding to $\sigma_2$ is
\[
\Gamma_2 \eqd \big\{ x^{m_1}y^{m_2}z^{m_3} \in \lau \st \phi_{3}(x^{m_1}y^{m_2}z^{m_3}) \in \Gamma'_2 \big\}
\]
where the right round down function $\phi_3$ is
\[
\phi_{3}(x^{m_1}y^{m_2}z^{m_3})=\alpha^{m_1}\beta^{m_2}\gamma^{\lf\frac{1}{20}m_1 + \frac{3}{20} m_2 + \frac{4}{20}m_3\rf}.
\]
Thus 
\[
\Gamma_2=\left\{
\begin{matrix}
y^3z^{-2} & y^3z^{-1} & y^3 & y^3z &y^3z^2 \\
&y^2z^{-1} & y^2 & y^2z &y^2z^2&y^2z^3 \\
&&y & yz &yz^2&yz^3&yz^4 \\
&&1 & z &z^2&z^3&z^4
\end{matrix}
\right\}\!.
\]

Note that $S(\Gamma_1)=\sigma_1\dual\cap M$ and $S(\Gamma_2)=\sigma_2\dual\cap M$.

Now we turn to stability parameters. Since $Y_k$ is $G_k\HHilb$ for each $k=2,3$, from~(\ref{Eqtn:Stab for G-Hilb}) we can take
\[
\theta^{(2)}=(-2,1,1), \quad \theta^{(3)}=(-3,1,1,1).
\]
Then the condition~(\ref{Eqtn:partial solution}) of $\theta_P$ for given $\theta^{(2)}$ and $\theta^{(3)}$ is
\[
\left\{
\begin{array}{ccl}
-2 & =& \sum_{l=0}^6 \theta_P(\rho_{3l}),\\[6pt]
1 & =& \sum_{l=0}^6 \theta_P(\rho_{3l+1}),\\[6pt]
1 & =& \sum_{l=0}^5 \theta_P(\rho_{3l+2}),\\[6pt]
-3 & =& \sum_{l=0}^4 \theta_P(\rho_{4l}),\\[6pt]
1 & =& \sum_{l=0}^4 \theta_P(\rho_{4l+1}),\\[6pt]
1 & =& \sum_{l=0}^4 \theta_P(\rho_{4l+2}),\\[6pt]
1 & =& \sum_{l=0}^4 \theta_P(\rho_{4l+3}).\\[6pt]
\end{array}
\right.
\]
Take 
\[
\theta_P=(-3,0,0,0,0,1,1,1,0,\ldots,0)
\]
as a solution of the equations above. For $\vtheta$ in~(\ref{Eqtn:def. of vartheta for the 1st case}), define $\theta=\theta_P+m\vtheta$:
\[
\theta (\rho_i)=(\theta_P+m\vtheta)(\rho_i) =
\begin{cases}
-3-m &\text{if $i=0$,}\\
-m &\text{if $1\leq i \leq 3$,}\\
0 &\text{if $i=4$,}\\
1 &\text{if $i=5 \mbox{ or }6$,}\\
1-m &\text{if $i=7$,}\\
m &\text{if $15\leq i \leq 19$,}\\
0 &\text{otherwise.}\\
\end{cases}
\]

Consider the $G$-brick $\Gamma_2$ above:
\[
\Gamma_2=\left\{
\begin{matrix}
y^3z^{-2} & y^3z^{-1} & y^3 & y^3z &y^3z^2 \\
&y^2z^{-1} & y^2 & y^2z &y^2z^2&y^2z^3 \\
&&y & yz &yz^2&yz^3&yz^4 \\
&&1 & z &z^2&z^3&z^4
\end{matrix}
\right\}\!.
\]
As examples, consider the two submodules $\sG$, $\sH$ generated by $A$ and $B$, respectively, where
\[
A=\left\{
\begin{matrix}
y^3z^{-2} & y^3z^{-1} & y^3 & y^3z &y^3z^2 \\
&y^2z^{-1} & y^2 & y^2z &y^2z^2&y^2z^3
\end{matrix}
\right\}\!,
\]
\[
B=\left\{
\begin{matrix}
y^3 & y^3z &y^3z^2 \\
y^2 & y^2z &y^2z^2&y^2z^3 \\
y & yz &yz^2&yz^3&yz^4 \\
1 & z &z^2&z^3&z^4
\end{matrix}
\right\}\!.
\]
First consider the submodule $\sG$. Note that $\vtheta(\sG)=0$. By definition, note that $\phi_3(A)=\{\beta^2, \beta^3\}$ forms a basis of a submodule $\sG'$ of $C(\Gamma_2')$ with $\theta(\sG)=\theta^{(3)}(\sG')$. Thus
\[
\theta(\sG)=\theta^{(3)}(\sG')=2>0.
\]
For the submodule $\sH'$, note that $\phi_3\inv\big(\phi_3(B)\big)$ contains $y^2z^{-1}$, $y^3z^{-1}$ and $y^3z^{-2}$. Observe that $\vtheta(\sH)>0$. Thus $\theta(\sH)$ is positive for large enough $m$. More precisely,
\[
\theta(\sH)=-3+1+1+m+m=2m-1
\]
is positive if $m>\frac{1}{2}$.
\eeg

\subsection{\for{toc}{The second case: ${r=abc+a-2b+1}$}\except{toc}{The second case: \boldmath${r=abc+a-2b+1}$}}\label{Sec:2nd case}
Consider the group of type $\frac{1}{r}(1,a,b)$. 
Assume that the star subdivision at $v=\frac{1}{r}(1,a,b)$ gives:
\begin{enumerate}
\item $\sigma_2:=\Cone(e_1,v,e_3)$ is of type $\frac{1}{a}(1,1,1)$ for $a \geq 4$;
\item $\sigma_3:=\Cone(e_1,e_2,v)$ is a Gorenstein quotient singularity.
\end{enumerate}
This means that:
\begin{enumerate}
\item $-r \equiv 1 \mod a$;
\item $1-r+b \equiv 3 \mod a$;
\item $1-r+a \equiv 0 \mod b$.
\end{enumerate}

In the rest of this section, we consider the case where
\[
r=abc-2b+a+1 \quad \text{with} \quad b=ak+1, a\geq 4
\]
for some positive integers $c,k$. Consider the lattice
\[
L = \Z^3 + \Z\cdot \frac{1}{r}(1,a,b),
\]
Let $v$ and $w$ denote the lattice points
\[
v:=\frac{1}{r}(1,a,b)\quad\text{and}\quad w:=\frac{1}{r}(\frac{r+1}{a}, 1, \frac{r+b}{a} ).
\]

Let $X_v$ denote the toric variety corresponding to the star subdivision at $v$. In this case, $X_v$ is not the relative canonical model of $X=\C^3/G$ because the quotient of type $\frac{1}{a}(1,1,1)$ is not canonical for $a \geq 4$. The relative canonical model depends on $c$. We have the two cases:
\begin{enumerate}
\item[(a)]$c\geq 2$;
\item[(b)]$c=1$.
\end{enumerate}
\subsubsection*{Case (a): $c\geq 2$}
Consider the case where $c \geq 2$. In this case, the relative canonical model is given by the fan consisting of the following five cones and their faces:
\[
\begin{array}{lll}
\sigma_1=\Cone (v,e_2, e_3), \ & \sigma_3=\Cone (e_1, e_2, v),\ & \\[2pt]
\sigma_4=\Cone (w, v, e_3), \ & \sigma_6=\Cone (e_1, v, w),\ & \sigma_7=\Cone (e_1, w, e_3).
\end{array}
\]
\begin{figure}
\begin{center}
\begin{tikzpicture}
\coordinate [label=left:$e_3$] (e3) at (0,0);
\coordinate [label=right:$e_2$] (e2) at (8,0);
\coordinate [label=right:$e_1$] (e1) at (4,6.8);
\coordinate (v1) at (3,1);
\draw[fill] (v1) circle [radius=0.05];
\draw[fill] (e3) circle [radius=0.05];
\draw[fill] (e2) circle [radius=0.05];
\draw[fill] (e1) circle [radius=0.05];
\node [above right] at (v1) {$v$};
\coordinate [label=left:$w$] (v7) at (2.2,1.9);
\draw[fill] (v7) circle [radius=0.05];

\draw (e3) -- (e2);
\draw (e3) -- (e1);
\draw (e2) -- (e1);
\draw (e3) -- (v1);
\draw (e2) -- (v1);
\draw (v1) -- (e1);
\draw (v7) -- (e1);
\draw (v7) -- (v1);
\draw (v7) -- (e3);
\node [below] at (4.2,0.6) {$\sigma_{1}$};
\node  at (4.6,3) {$\sigma_{3}$};
\node  at (2,1) {$\sigma_{4}$};
\node  at (2.85,1.9) {$\sigma_{6}$};
\node  at (2,2.7) {$\sigma_{7}$};

\end{tikzpicture}
\end{center}
\caption{Canonical model for $c\geq 2$}
\label{Fig:Canonical model for c geq 2}
\end{figure}
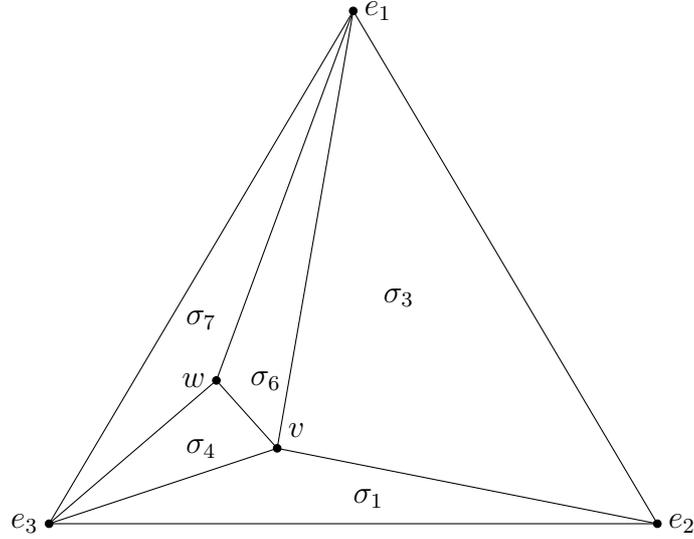
Indeed, the cone $\sigma_2$ defines a Gorenstein quotient singularity and the others define smooth affine toric open sets. We can check directly $\KXcan$ is ample over $X$.

Since there exists a projective morphism $X_{\mathrm{can}}\to X_v$, for every relative minimal model $\varphi \colon Y\to X$, we have a projective morphism $\overline{{\varphi}} \colon Y\to X_v$ fitting into:
\[ 
\begin{tikzcd} 
Y \arrow{r}\arrow{rd}{\overline{{\varphi}}} \arrow{rdd}{{\varphi}}& X_{\mathrm{can}} \arrow{d}\\
& X_{v}\arrow{d}{\nu} \\
&X.
\end{tikzcd}\]

The morphism $\overline{\varphi}$ induces two projective morphisms:
\begin{enumerate}
\item $\overline{\varphi}_2\colon Y_2 \to U_2=\C^3/G_2$;
\item $\overline{\varphi}_3\colon Y_3 \to U_3=\C^3/G_3$.
\end{enumerate}
As is seen above, $G_2$ is of type $\frac{1}{a}(1,1,1)$ and the induced morphism $\overline{\varphi}_2$ is given by
\[
\GHii{2}\to \C^3/G_2
\]
where $G_2$ is of type $\frac{1}{a}(1,1,1)$. Thus it follows that there exist a brickset $\gr_2$ for $Y_2\to U_2$ and $\theta^{(2)}$ for the brickset $\gr_2$. On the other hand, since $U_3$ is a Gorenstein quotient singularity, by the Craw--Ishii Theorem\cite{CI}, there exists a brickset $\gr_3$ for $Y_3\to U_3$ and $\theta^{(3)}$ for the brickset $\gr_3$.

From Theorem~\ref{Thm:subdivision and brickset}, there is a $G$-brickset for $Y\to X$. Now it suffices to find a GIT parameter $\theta$ such that every $\Gamma \in \gr$ is $\theta$-stable.
Define the GIT parameter $\vtheta\in\Theta$ by
\[
\vtheta(\rho) =
\begin{cases}
-1 &\text{if $0 \leq \wt(\rho)<b$,}\\
-1 &\text{if $\wt(\rho)=2ab-5b+3$,}\\
1 &\text{if $\wt(\rho)=r-a-b+2$,}\\
1 &\text{if $r-b \leq \wt(\rho)<r$,}\\
0 &\text{otherwise.}\\
\end{cases}
\]
Note that $\vtheta$ above has the same properties in Remark~\ref{Rem:key properties of vartheta}. Thus the same proof works for the existence of $\theta$ as in Proposition~\ref{Prop:our stability for 1st case}. Therefore the following theorem follows.
\begin{Thm}\label{Thm:Main Theorem:2nd case c geq 2}
Consider positive integers $a,k,c$ with $c\geq 2$, $a\geq 4$ and $b=ak+1$. Let $G$ be the group of type $\frac{1}{r}(1,a,b)$ with $r=abc+a-2b+1$. Let $Y\rightarrow X:=\C^3/G$ be a relative minimal model of $X$. 
Then $Y$ is isomorphic to the birational component $\yth$ of the moduli space $\mth$ of $\theta$-stable $G$-constellations for a suitable parameter~$\theta$.
\end{Thm}

\subsubsection*{Case (b): $c=1$}
For the case where $c=1$, the fan of the relative canonical model consists of the following four cones and their faces:
\[
\begin{array}{lll}
\sigma_1=\Cone (v,e_2, e_3), \quad & \sigma_5=\Cone (e_1, e_2, v, w),\quad & \\[2pt]
\sigma_4=\Cone (w, v, e_3),\quad & \sigma_7=\Cone (e_1, w, e_3).
\end{array}
\]
\begin{figure}
\begin{center}
\begin{tikzpicture}
\coordinate [label=left:$e_3$] (e3) at (0,0);
\coordinate [label=right:$e_2$] (e2) at (8,0);
\coordinate [label=right:$e_1$] (e1) at (4,6.8);
\coordinate (v1) at (3,1);
\draw[fill] (v1) circle [radius=0.05];
\draw[fill] (e3) circle [radius=0.05];
\draw[fill] (e2) circle [radius=0.05];
\draw[fill] (e1) circle [radius=0.05];
\node [above right] at (v1) {$v$};
\coordinate [label=left:$w$] (v7) at (2.2,1.9);
\draw[fill] (v7) circle [radius=0.05];

\draw (e3) -- (e2);
\draw (e3) -- (e1);
\draw (e2) -- (e1);
\draw (e3) -- (v1);
\draw (e2) -- (v1);
\draw (v7) -- (e1);
\draw (v7) -- (v1);
\draw (v7) -- (e3);
\node [below] at (4.2,0.6) {$\sigma_{1}$};
\node  at (4.3,2.8) {$\sigma_{5}$};
\node  at (2,1) {$\sigma_{4}$};
\node  at (2,2.7) {$\sigma_{7}$};

\end{tikzpicture}
\end{center}
\caption{Canonical model for $c= 1$}
\label{Fig:Canonical model for c = 1}
\end{figure}
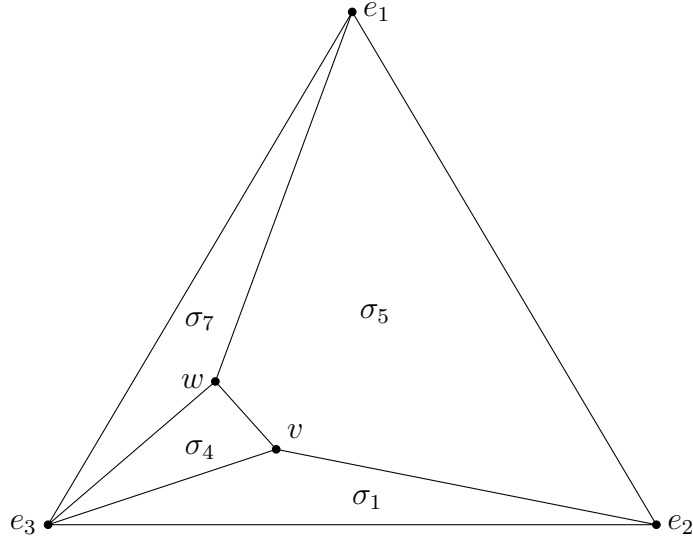
Indeed, the cone $\sigma_5$ defines a toric Gorenstein singularity and hence it is canonical.
Note that since the cone $\sigma_5$ is not simplicial, the corresponding affine toric variety is not a quotient type. In particular, the relative canonical model does not need to be obtained by a sequence of star subdivisions because a star subdivision of a simplicial fan is simplicial.

Note that since $X_{\mathrm{can}}$ is Gorenstein, every relative minimal model is smooth. However, some relative minimal of $X$ does not have a morphism to $X_v$ (See Example~\ref{Eg:1/39(1,5,11):no morphism to X_v}). For a relative minimal model $Y$ admitting a morphism to $X_v$ and for $a\geq 6$, we can prove in the same way as in Case~(a) with the following $\vtheta$:
\[
\vtheta(\rho) =
\begin{cases}
-1 &\text{if $0 \leq \wt(\rho)<b$,}\\
-1 &\text{if $\wt(\rho)=ab-5b+3$,}\\
1 &\text{if $\wt(\rho)=r-a-b+2$,}\\
1 &\text{if $r-b \leq \wt(\rho)<r$,}\\
0 &\text{otherwise.}\\
\end{cases}
\]
\begin{Prop}
For positive integers $a,k$, let $G$ be the group of type $\frac{1}{r}(1,a,b)$ with $r=ab+a-2b+1$ and $b=ak+1$. Furthermore assume that $a\geq 6$. Let $X_v$ denote the toric variety given by the star subdivision of $\sigma_+$ at $v=\frac{1}{r}(1,a,b)$. Let $Y\rightarrow X:=\C^3/G$ be a relative minimal model admitting a morphism $Y\to X_v$.
Then $Y$ is isomorphic to the birational component $\yth$ of the moduli space $\mth$ of $\theta$-stable $G$-constellations for a suitable parameter~$\theta$.
\end{Prop}

\begin{Eg}\label{Eg:1/39(1,5,11):no morphism to X_v}
Consider the group $G$ of type $\frac{1}{39}(1,5,11)$. Then the star subdivision at $v=\frac{1}{39}(1,5,11)$ gives:
\begin{enumerate}
\item $\sigma_2:=\Cone(e_1,v,e_3)$ corresponds to the quotient singularity of type $\frac{1}{5}(1,1,1)$;
\item $\sigma_3:=\Cone(e_1,e_2,v)$ corresponds to the quotient singularity of type $\frac{1}{11}(1,5,5)$.
\end{enumerate}

\begin{figure}
\begin{center}
\begin{tikzpicture}

\coordinate [label=below:$e_3$] (se3) at (3-3.30,0);
\coordinate [label=below:$e_2$] (se2) at (9-3.30,0);
\coordinate [label=above:$e_1$] (se1) at (6-3.30,5.1);
\draw[fill] (se1) circle [radius=0.04];
\draw[fill] (se2) circle [radius=0.04];
\draw[fill] (se3) circle [radius=0.04];
\draw (se3) -- (se2);
\draw (se3) -- (se1);
\draw (se2) -- (se1);

\coordinate [label=below:$v_1$] (sv1) at (5-3.30,0.7);
\coordinate [label=right:$v_8$] (sv8) at (4.7-3.30,1.79);
\foreach \x in {1,8}
\draw[fill] (sv\x) circle [radius=0.04];
\foreach \x in {1,2,3}
\draw (sv1) -- (se\x);

\foreach \k in {0,1,2,3,4}
\coordinate (svv\k) at (7-\k*0.2-3.30, 1.4+\k*0.74);
\foreach \k in {0,1,2,3,4}
\draw[fill] (svv\k) circle [radius=0.04];

\node [below] at (svv0) {$v_{4}$};
\node [below] at (svv1) {$v_{11}$};
\node [below] at (svv2) {$v_{18}$};
\node [below] at (svv3) {$v_{25}$};
\node [below] at (svv4) {$v_{32}$};

\node [right] at (3.4-3.30-0.3,-0.6-0.3) {$X_v:\text{the star subdivision at $v_1$}$};

\coordinate [label=below:$e_3$] (e3) at (3+3.30,0);
\coordinate [label=below:$e_2$] (e2) at (9+3.30,0);
\coordinate [label=above:$e_1$] (e1) at (6+3.30,5.1);
\draw[fill] (e1) circle [radius=0.04];
\draw[fill] (e2) circle [radius=0.04];
\draw[fill] (e3) circle [radius=0.04];
\draw (e3) -- (e2);
\draw (e3) -- (e1);
\draw (e2) -- (e1);

\coordinate [label=below:$v_1$] (v1) at (5+3.30,0.7);
\coordinate [label=right:$v_8$] (v8) at (4.7+3.30,1.79);
\foreach \x in {1,8}
\draw[fill] (v\x) circle [radius=0.04];
\foreach \x in {2,3}
\draw (v1) -- (e\x);
\draw (v1) -- (v8);
\foreach \x in {1,3}
\draw (v8) -- (e\x);

\foreach \k in {0,1,2,3,4}
\coordinate (vv\k) at (7-\k*0.2+3.30, 1.4+\k*0.74);
\foreach \k in {0,1,2,3,4}
\draw[fill] (vv\k) circle [radius=0.04];

\node [below] at (vv0) {$v_{4}$};
\node [below] at (vv1) {$v_{11}$};
\node [below] at (vv2) {$v_{18}$};
\node [below] at (vv3) {$v_{25}$};
\node [below] at (vv4) {$v_{32}$};

\node [right] at (3.4+3.30-0.3-0.2,-0.6-0.3) {$X_{\mathrm{can}}: \text{a relative canonical model}$};

\coordinate [label=below:$e_3$] (le3) at (3-3.30,0+6.50);
\coordinate [label=below:$e_2$] (le2) at (9-3.30,0+6.50);
\coordinate [label=above:$e_1$] (le1) at (6-3.30,5.1+6.50);
\draw[fill] (le1) circle [radius=0.04];
\draw[fill] (le2) circle [radius=0.04];
\draw[fill] (le3) circle [radius=0.04];
\draw (le3) -- (le2);
\draw (le3) -- (le1);
\draw (le2) -- (le1);

\coordinate [label=below:$v_1$] (lv1) at (5-3.30,0.7+6.50);
\coordinate [label=left:$v_8$] (lv8) at (4.7-3.30,1.79+6.50);
\foreach \x in {1,8}
\draw[fill] (lv\x) circle [radius=0.04];
\foreach \x in {2,3}
\draw (lv1) -- (le\x);
\draw (lv1) -- (lv8);
\foreach \x in {1,3}
\draw (lv8) -- (le\x);

\foreach \k in {0,1,2,3,4}
\coordinate (lvv\k) at (7-\k*0.2-3.30, 1.4+\k*0.74+6.50);
\foreach \k in {0,1,2,3,4}
\draw[fill] (lvv\k) circle [radius=0.04];

\node [below left] at (lvv0) {$v_{4}$};
\node [below left] at (lvv1) {$v_{11}$};
\node [below left] at (lvv2) {$v_{18}$};
\node [below left] at (lvv3) {$v_{25}$};
\node [below left] at (lvv4) {$v_{32}$};

\node [right] at (3.4-3.50,-0.6+6.50) {$Y: \text{a relative minimal model}$};


\draw (lv1) -- (le1);
\foreach \k in {0,1,2,3,4}
\draw (lv1) -- (lvv\k);
\foreach \k in {0,1,2,3,4}
\draw (le2) -- (lvv\k);
\draw (le1) -- (lvv0);

\coordinate [label=below:$e_3$] (re3) at (3+3.30,0+6.50);
\coordinate [label=below:$e_2$] (re2) at (9+3.30,0+6.50);
\coordinate [label=above:$e_1$] (re1) at (6+3.30,5.1+6.50);
\draw[fill] (re1) circle [radius=0.04];
\draw[fill] (re2) circle [radius=0.04];
\draw[fill] (re3) circle [radius=0.04];
\draw (re3) -- (re2);
\draw (re3) -- (re1);
\draw (re2) -- (re1);

\coordinate [label=below:$v_1$] (rv1) at (5+3.30,0.7+6.50);
\coordinate [label=left:$v_8$] (rv8) at (4.7+3.30,1.79+6.50);
\foreach \x in {1,8}
\draw[fill] (rv\x) circle [radius=0.04];
\foreach \x in {2,3}
\draw (rv1) -- (re\x);
\draw (rv1) -- (rv8);
\foreach \x in {1,3}
\draw (rv8) -- (re\x);

\foreach \k in {0,1,2,3,4}
\coordinate (rvv\k) at (7-\k*0.2+3.30, 1.4+\k*0.74+6.50);
\foreach \k in {0,1,2,3,4}
\draw[fill] (rvv\k) circle [radius=0.04];

\node [below left] at (rvv0) {$v_{4}$};
\node [below left] at (rvv1) {$v_{11}$};
\node [below left] at (rvv2) {$v_{18}$};
\node [below left] at (rvv3) {$v_{25}$};
\node [below left] at (rvv4) {$v_{32}$};

\node [right] at (3.4+3.50,-0.6+6.50) {$Z: \text{a relative minimal model}$};


\draw (rv8) -- (rvv0);
\draw (rv1) -- (rvv0);
\foreach \k in {0,1,2,3,4}
\draw (rv8) -- (rvv\k);
\foreach \k in {0,1,2,3,4}
\draw (re2) -- (rvv\k);
\draw (re1) -- (rvv0);

\end{tikzpicture}
\end{center}
\caption{Fans of birational models for $\frac{1}{39}(1,5,11)$}
\label{Fig:1/39(1,5,11)}
\end{figure}
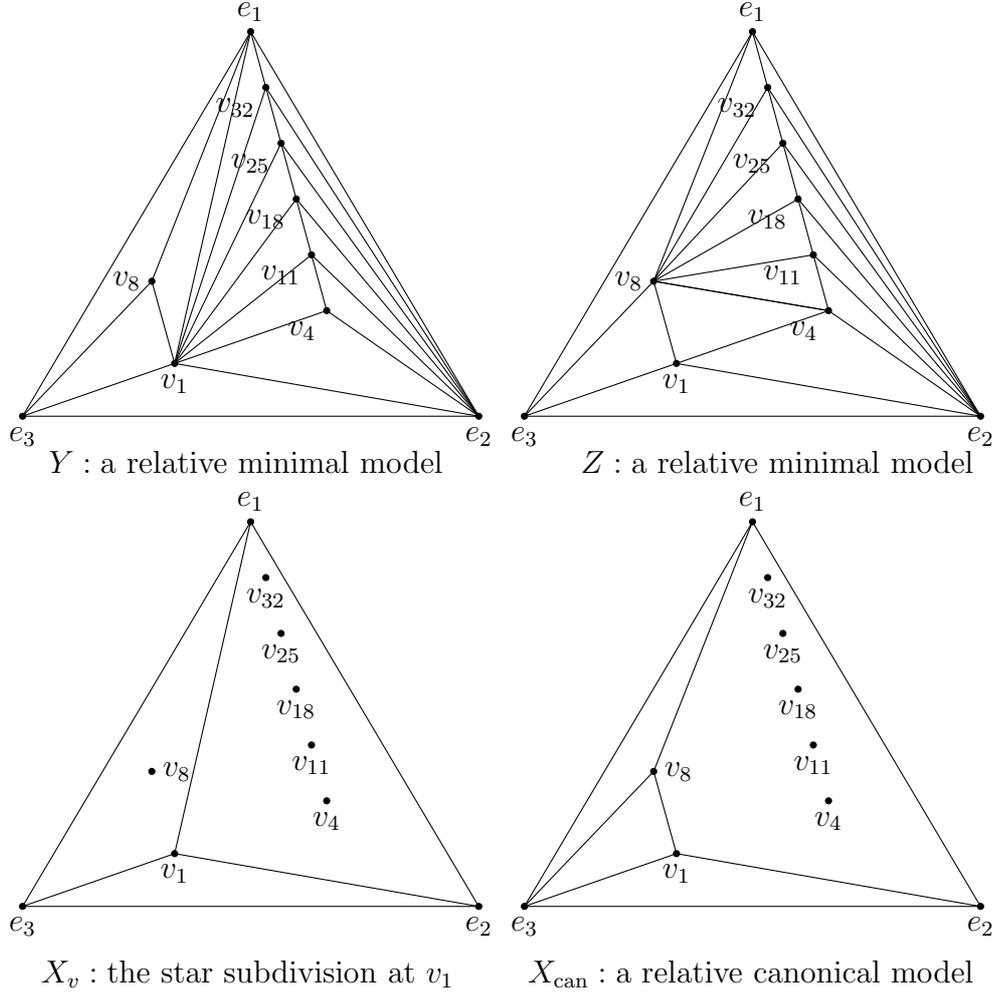

As is discussed above, the relative canonical model $X_{\mathrm{can}}$ of $X=\C^3/G$ is Gorenstein, but not $\Q$-factorial. 

Let $v_i$ denote the lattice point $\frac{1}{r}(\overline{i}, \overline{5i}, \overline{11i})$ where $\bar{\quad}$ denotes the residue modulo $r$. In particular, $v_1=v$ and $v_8=w$. Note that there exists a plane $\Pi$ containing $e_1$, $e_2$, $v_1$ and $v_8$. Observe that the lattice points $v_4$, $v_{11}$, $v_{18}$, $v_{25}$, and $v_{32}$ lie on the plane $\Pi$. Thus subdividing the cone $\sigma_5$ into smooth cones only using these points defines a crepant resolution of the toric singularity given by $\sigma_5$ where
\[
\sigma_5=\Cone(e_1,e_2,v_1,v_8).
\]

In Figure~\ref{Fig:1/39(1,5,11)}, the variety $Y$ is a relative minimal model of $X$ admitting a morphism to $X_v$. Actually one can prove that $Y$ is isomorphic to $\yth$ for some $\theta$. On the other hand, the variety $Z$ is a relative minimal model having no morphism to $X_v$. At this moment, we do not know that whether $Z$ is isomorphic to $\yth$ for some~$\theta$.
\eeg
\newpage
\subsection{Discussions} 
\subsubsection{Smoothness of minimal models} From Theorem~\ref{Thm:Toric MMP}, it follows that if the relative canonical model $X_{\mathrm{can}}$ is Gorenstein, then any relative minimal model is Gorenstein. Since a toric Gorenstein 3-fold terminal singularity is smooth, every relative minimal model is smooth if the relative canonical model $X_{\mathrm{can}}$ is Gorenstein for $G$ being abelian. However, we do not know any sufficient condition for the group $G$ of type $\frac{1}{r}(1,a,b)$ having the Gorenstein relative canonical model $X_{\mathrm{can}}$ of $\C^3/G$.
\begin{Que}
Let $G$ be the group of type $\frac{1}{r}(1,a,b)$ with $a+b+1<r$. Let $X$ be the quotient $\C^3/G$ and $X_{\mathrm{can}}$ the relative canonical model of $X$. When is $X_{\mathrm{can}}$ Gorenstein? If so, when can we obtain $X_{\mathrm{can}}$ by a sequence of star subdivisions?
\end{Que}
\subsubsection{Other stability parameters}
The theorems above said that for a relative minimal model $Y$ there exists some parameter $\theta$ such that $\yth$ is isomorphic to $Y$. We can ask whether $\yth$ is a relative minimal model for all generic $\theta$ or not. 

Sara Muhvi\'{c} calculated the following:
\begin{enumerate}
\item for the type of $\frac{1}{12}(1,2,3)$, $\GHilb{3}$ is smooth but not a relative minimal model;
\item for the type of $\frac{1}{24}(1,3,5)$, $\GHilb{3}$ is not even smooth.
\end{enumerate}
Thus it seems that there exist few chambers in $\Theta$ giving a relative minimal model of $\C^3/G$.
\begin{Que}
Let $G$ be the group of type in Section~\ref{Sec:1st case} or Section~\ref{Sec:2nd case}. For which $\theta$, is $\yth$ a relative minimal model?
\end{Que}

\subsubsection{Existence of stability parameters} Let $Y\to X=\C^3/G$ be a relative minimal model admitting a morphism to $X_v$ where $X_v$ is the toric variety given by the star subdivision of $\sigma_+$ at $v$. The main theorem was proved by showing the three statements:
\begin{enumerate}
\item there exists a $G$-brickset $\gr$ for $Y\to X$ using round down functions;
\item the linear map $\Theta \to \Theta^{(1)}\oplus\Theta^{(2)}\oplus\Theta^{(3)}$ is surjective in~(\ref{Eqtn:partial solution});
\item there exists a stability parameter $\vtheta$ satisfying~(\ref{Rem:key properties of vartheta}).
\end{enumerate} 

To prove (i), we only used the existence of a $G_k$-brickset for $G_k$, whose order is smaller than that of $G$. When we proved (ii), we only use the assumption that $a$ and $b$ are coprime. However, showing the existence of $\vtheta$ in~(iii) was done on a case by case basis in Section~\ref{Sec:1st case} and Section~\ref{Sec:2nd case}. It would be interesting if we have a systematic way to produce such a parameter~$\vtheta$.
\begin{Que}
Is there a systematic method to find a stability parameter $\vtheta$ satisfying the properties in Remark~\ref{Rem:key properties of vartheta} for a star subdivision?
\end{Que}
\appendix
\section{$\frac{1}{39}(1,5,11)$ type}
Let $G$ be the group of type $\frac{1}{39}(1,5,11)$ as in Example~\ref{Eg:1/39(1,5,11):no morphism to X_v}. Consider the relative minimal model $Z$ in Figure~\ref{Fig:1/39(1,5,11)}. In this section, although we cannot see that $Z$ is isomorphic to the birational component $\yth$ of $\mth$, we show that there exists a $G$-brickset for $Z\to X=\C^3/G$.

Although there is no morphism $Z\to X_v$, there exists a morphism $\overline{\varphi}\colon Z\to X_u$ where $X_u$ is the toric variety given by the star subdivision at $u=v_4=\frac{1}{39}(4,20,5)$. The star subdivision of $\sigma_+$ at $u$ produces the three cones:
\[
\sigma_1=\Cone(u,e_2,e_3),\quad\sigma_2=\Cone(e_1,e_2,u),\quad\sigma_3=\Cone(e_1,u,e_3).
\]
The morphism $\overline{\varphi}$ induces the following three morphisms:
\begin{enumerate}
\item $\overline{\varphi}_1\colon Z_1\to \C^3/G_1$, where $G_1$ is of type $\frac{1}{4}(1,0,1)$;
\item $\overline{\varphi}_2\colon Z_2\to \C^3/G_2$, where $G_2$ is of type $\frac{1}{20}(4,1,5)$;
\item $\overline{\varphi}_3\colon Z_3\to \C^3/G_3$, where $G_3$ is of type $\frac{1}{5}(4,0,1)$.
\end{enumerate}

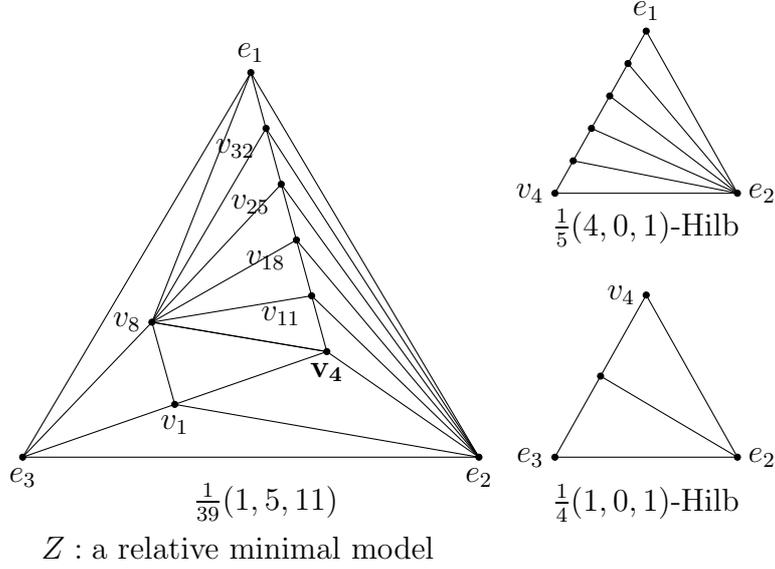
\begin{figure}
\begin{center}
\begin{tikzpicture}

\coordinate [label=below:$e_3$] (re3) at (3-2.70,0+6.50);
\coordinate [label=below:$e_2$] (re2) at (9-2.70,0+6.50);
\coordinate [label=above:$e_1$] (re1) at (6-2.70,5.1+6.50);
\draw[fill] (re1) circle [radius=0.04];
\draw[fill] (re2) circle [radius=0.04];
\draw[fill] (re3) circle [radius=0.04];
\draw (re3) -- (re2);
\draw (re3) -- (re1);
\draw (re2) -- (re1);

\coordinate [label=below:$v_1$] (rv1) at (5-2.70,0.7+6.50);
\coordinate [label=left:$v_8$] (rv8) at (4.7-2.70,1.79+6.50);
\foreach \x in {1,8}
\draw[fill] (rv\x) circle [radius=0.04];
\foreach \x in {2,3}
\draw (rv1) -- (re\x);
\draw (rv1) -- (rv8);
\foreach \x in {1,3}
\draw (rv8) -- (re\x);

\foreach \k in {0,1,2,3,4}
\coordinate (rvv\k) at (7-\k*0.2-2.70, 1.4+\k*0.74+6.50);
\foreach \k in {0,1,2,3,4}
\draw[fill] (rvv\k) circle [radius=0.04];

\node [below] at (rvv0) {$\mathbf{v_{4}}$};
\node [below left] at (rvv1) {$v_{11}$};
\node [below left] at (rvv2) {$v_{18}$};
\node [below left] at (rvv3) {$v_{25}$};
\node [below left] at (rvv4) {$v_{32}$};

\node [right] at (5.4-3.00,-0.6+6.50) {$\frac{1}{39}(1,5,11)$};
\node [right] at (3.4-3.00,-0.6+5.90) {$Z: \text{a relative minimal model}$};


\draw (rv8) -- (rvv0);
\draw (rv1) -- (rvv0);
\foreach \k in {0,1,2,3,4}
\draw (rv8) -- (rvv\k);
\foreach \k in {0,1,2,3,4}
\draw (re2) -- (rvv\k);
\draw (re1) -- (rvv0);


\coordinate [label=left:$v_4$] (ue3) at (7.3,0+10);
\coordinate [label=right:$e_2$] (ue2) at (9.7,0+10);
\coordinate [label=above:$e_1$] (ue1) at (8.5,2.15+10);
\draw[fill] (ue1) circle [radius=0.04];
\draw[fill] (ue2) circle [radius=0.04];
\draw[fill] (ue3) circle [radius=0.04];
\draw (ue3) -- (ue2);
\draw (ue3) -- (ue1);
\draw (ue2) -- (ue1);

\foreach \k in {1,2,3,4}
\coordinate (uv\k) at (7.3+\k*0.24, 10+\k*0.43);

\foreach \k in {1,2,3,4}
\draw[fill] (uv\k) circle [radius=0.04];

\foreach \k in {1,2,3,4}
\draw (ue2) -- (uv\k);

\node [above] at (8.5,2.15+7.00) {$\frac{1}{5}(4,0,1)\HHilb$};

\coordinate [label=left:$e_3$] (ue3) at (7.3,0+6.50);
\coordinate [label=right:$e_2$] (ue2) at (9.7,0+6.50);
\coordinate [label=left:$v_4$] (ue1) at (8.5,2.15+6.50);
\draw[fill] (ue1) circle [radius=0.04];
\draw[fill] (ue2) circle [radius=0.04];
\draw[fill] (ue3) circle [radius=0.04];
\draw (ue3) -- (ue2);
\draw (ue3) -- (ue1);
\draw (ue2) -- (ue1);

\coordinate (dv1) at (7.3+2.5*0.24, 6.50+2.5*0.43);

\draw[fill] (dv1) circle [radius=0.04];

\draw (ue2) -- (dv1);

\node [above] at (8.5,2.15+3.35) {$\frac{1}{4}(1,0,1)\HHilb$};



%
%






\end{tikzpicture}
\end{center}
\caption{Recursion process for $Z$}
\label{Fig:Recursion for Z}
\end{figure}

As is shown in Figure~\ref{Fig:Recursion for Z}, note that for $k=1,3,$ the morphism $\overline{\varphi_k}\colon Z_k\to\C^3/G_k$ is given by
\[
G_k\HHilb\C^3 \to \C^3/G_k.
\]
Thus, to show the existence of a $G$-brickset $\gr$ for $Z\to X$, it only remains to show there exists a $G_2$-brickset for $\overline{\varphi_2}\colon Z_2 \to \C^3/G_2$ by Theorem~\ref{Thm:subdivision and brickset}. Considering the star subdivision of $\Cone(e_1,v_4,e_3)$ at $v_8$, one can see that there exists a stability parameter $\theta^{(2)}$ such that $Z_2$ is isomorphic to the birational component of the moduli space of $\theta^{(2)}$-stable $G_2$-constellations in a similar way to the case in the main theorem. Therefore we can conclude that there exists a $G$-brickset\footnote{You can find the $G$-brickset $\gr$ on my website:\\ \url{http://newton.kias.re.kr/~seungjo/CI1.html}} $\gr$ for $Z\to X$.

Finally, we discuss why we cannot see the existence of $\theta$.
First, a parameter $\vtheta$ satisfying~(\ref{Rem:key properties of vartheta}) can be found, eg.\ $\vtheta$ can be defined to be
\[
\vtheta(\rho_i) =
\begin{cases}
-1 &\text{if $0 \leq i \leq 18$,}\\
1 &\text{if $i=19$,}\\
1 &\text{if $20 \leq i \leq 38$.}
\end{cases}
\]
On the other hand, the linear map
\[
\phi_{\star}=\left((\phi_1)_{\star},(\phi_2)_{\star},(\phi_3)_{\star}\right)\colon \Theta \to \Theta^{(1)}\oplus\Theta^{(2)}\oplus \Theta^{(3)}
\]
is not surjective. Therefore, we cannot tell if there exists a solution for~(\ref{Eqtn:partial solution}).
However, this does not mean that there are no parameters $\theta$ for the $G$-brickset $\gr$. Using a computer, we might be able to find a parameter $\theta$ such that every $\Gamma \in \gr$ is $\theta$-stable. 


\bibliographystyle{alpha}

\begin{thebibliography}{100}
\bibitem
{BKR} T. Bridgeland, A. King, M. Reid,
{\it The McKay correspondence as an equivalence of derived categories},
J. Amer. Math. Soc. {\bf 14} (2001), no. 3, 535--554.





\bibitem
{CI} A. Craw, A. Ishii,
{\it Flops of $\GHil$ and equivalences of derived categories by variation of GIT quotient},
Duke Math. J. {\bf 124} (2004), no. 2, 259--307.

\bibitem
{CLS} D. Cox, J. Little, H. Schenck,
{\it Toric Varieties}, Graduate Studies in Mathematics, {\bf 124}.
American Mathematical Society, Providence, RI, 2011.

\bibitem
{CMT} A. Craw, D. Maclagan, R. R. Thomas,
{\it Moduli of McKay quiver representations I: The coherent component},
Proc. Lond. Math. Soc. (3) {\bf 95} (2007), no. 1, 179--198.

\bibitem
{CMTb} A. Craw, D. Maclagan, R. R. Thomas,
{\it Moduli of McKay quiver representations II: Gr\"{o}bner basis techniques},
J. Algebra {\bf 316} (2007), no. 2, 514--535



\bibitem
{DLR} S. Davis, T. Logvinenko, M. Reid,
{\it How to calculate $\AHilb \C^n$ for $\frac{1}{r}(a,b,1,\ldots,1)$},
preprint.


\bibitem
{IN} Y. Ito, H. Nakajima,
{\it McKay correspondence and Hilbert schemes in dimension three},
Topology {\bf 39} (2000), no. 6, 1155--1191.

\bibitem
{INa} Y. Ito, I. Nakamura,
{\it Hilbert schemes and simple singularities},
New trends in algebraic geometry (Warwick, 1996), 151--233, 
London Math. Soc. Lecture Note Ser., {\bf 264}, 
Cambridge Univ. Press, Cambridge, 1999.

\bibitem
{Thesis} S.-J. Jung, 
{\it McKay Quivers and Terminal Quotient Singularities in Dimension~3}, PhD thesis, University of Warwick, 2014.

\bibitem
{Terminal} S.-J. Jung, 
{\it Terminal quotient singularities in dimension three via variation of GIT}, in preprint, arxiv:1502.03579.





\bibitem
{K94} A. King,
{\it Moduli of representations of finite dimensional algebras},
Quart. J. Math. Oxford Ser.(2) {\bf 45} (1994), no. 180, 515--530.











\bibitem
{N01} I. Nakamura,
{\it Hilbert schemes of abelian group orbits},
J. Algebraic. Geom. {\bf 10} (2001), no.4, 757--779.




\bibitem
{Rtoric} M. Reid,
{\it Decomposition of toric morphisms},
Arithmetic and geometry, Vol. II, 395--418, Progr. Math., {\bf 36}, Birkh\"{a}user, Boston, MA, 1983.

\bibitem
{R87} M. Reid,
{\it Young person's guide to canonical singularities},
Algebraic geometry, Bowdoin, 1985 (Brunswick, Maine, 1985), 345--414, Proc. Sympos. Pure Math., {\bf 46}, Part 1, Amer. Math. Soc., Providence, RI, 1987.





\end{thebibliography}


\end{document}